\documentclass[11pt, oneside]{article}

\usepackage{braket}

\usepackage{enumitem}

\usepackage{amsmath,amssymb,url}
\usepackage[mathscr]{euscript}
\usepackage{latexsym} 

\usepackage{amsthm}
\usepackage{amsfonts}

\usepackage[normalem]{ulem}

\usepackage{xcolor,cancel}

\newcommand\reallywidehat[1]{\arraycolsep=0pt\relax%
\begin{array}{c}
\stretchto{
  \scaleto{
    \scalerel*[\widthof{\ensuremath{#1}}]{\kern-.5pt\bigwedge\kern-.5pt}
    {\rule[-\textheight/2]{0.1ex}{\textheight}} 
  }{\textheight} %
}{0.5ex}\\           
#1\\                 
\rule{-1ex}{0ex}
\end{array}
}

\numberwithin{equation}{section}

\theoremstyle{plain}
\newtheorem{theorem}{Theorem}[section]
\newtheorem{lemma}[theorem]{Lemma}

\newtheorem{corollary}[theorem]{Corollary}

\theoremstyle{definition}
\newtheorem{definition}[theorem]{Definition}

\newtheorem{remark}[theorem]{Remark}

 \DeclareMathOperator{\Sk}{Sk}
 \DeclareMathOperator{\D}{D}
 
 \DeclareMathOperator{\diag}{diag}
 \DeclareMathOperator{\im}{im}
 
 \DeclareMathOperator{\C}{C} 
 \DeclareMathOperator{\Mod}{Mod}

\newcommand{\mbf}{\mathbf}

\newcommand{\Sg}[1]{\text{Sg}^{\mathbf{#1}}}

\newcommand{\longra}{\longrightarrow}
\newcommand{\lra}{\longrightarrow}
\newcommand{\ra}{\rightarrow}

\newcommand{\ol}{\overline}

\newcommand{\w}{\wedge}
\newcommand{\ora}{\overrightarrow}
\newcommand{\wh}{\widehat}
\newcommand{\wt}{\widetilde}
\newcommand{\bsup}{\dot{\vee}}
\newcommand{\Bsup}{\dot{\bigvee}}

\newcommand{\Q}{\mathbb{Q}}
\newcommand{\Z}{\mathbb{Z}}
\newcommand{\N}{\mathbb{N}}

\begin{document}

\author{Jo\"el
Adler\thanks{joel.adler@phbern.ch}
\\P\"adagogische Hochschule Bern, Switzerland}
\title{``The model companion of the class of
pseudocomplemented semilattices is
finitely axiomatizable'' revised}
\date{\today}
\maketitle \setlength{\parindent}{0pt} \setcounter{page}{1} 

\abstract{
It is shown that the class
$\mathcal{PCSL}^{ec}$ of existentially closed pseudocomplemented
semilattices is finitely axiomatizable by appropriately extending
a finite axiomatization of the class
$\mathcal{PCSL}^{ac}$ of algebraically closed
pseudocomplemented semilattices. Because
$\mathcal{PCSL}^{ec}$ coincides with the model companion of the class
$\mathcal{PCSL}$ of pseudocomplemented semilattices this ans\-wers the question asked by Albert and Burris in a paper in 1986: ``Does the class of
pseudocomplemented semilattices have a finitely axiomatizable model
companion?''
}

\maketitle


\setcounter{section}{-1}

\section{Changes concerning the author's paper Algebra Universalis (2014),\\ (DOI) 10.1007/s00012-014-0297-9
containing Lemma 5.3, which does not hold.}
%
%
The main theorem of the published version \cite[Theorem 5.9]{Ad} is split into a necessity part ---Theorem \ref{theorem_necessity}--- and a sufficiency part ---Theorem \ref{theorem_sufficiency}. The intermediate result in the proof of the sufficiency in the main theorem of the published version is now the independent Lemma \ref{lemma_criterion_ec}. This lemma is put at the beginning of Section \ref{sec_finite_ax_ec} because it determines the section.

The proof of Lemma \ref{lemma_criterion_ec} has to be carried out without using \cite[Theorem 5.3]{Ad} from the published version. Without this lemma the subalgebra $\mathbf{S}$ cannot be assumed to be isomorphic to a subdirectly irreducible p-semilattice $\mathbf{2}$ or $\widehat{\mathbf{F}_t}$, $t\geq 1$. It may still be assumed to be isomorphic to
a direct product of subdirectly irreducible p-semilattices $\mathbf{2}^{t}\times\prod_{i=1}^p\widehat{\mathbf{F}_{f(i)}}$, which is part of the statement of Lemma \ref{lemma_criterion_ec}. In the proof of this lemma the new Lemma \ref{lemma_fin_dir_prod} is used.

We have the following situation

\begin{equation}\label{eq_apprex}
      \mathbf{S}\cong\mathbf{2}^r\times\prod_{i=1}^{s_1}\mathbf{F}_{f_1(i)}\times\prod_{i=1}^{s_2}\widehat{\mathbf{F}_{f_2(i)}}\quad (r,s_1,s_2\in\N)
\end{equation}
%
\begin{equation}\label{eq_v1_vm}
\mathbf{T}\cong\mathbf{2}^{r'}\times\prod_{i=1}^{s'}\widehat{\mathbf{F}_{g(i)}} \quad(r',s', g(i)\in\N),
\end{equation}
with $r\leq r'$, $s_1+s_2\leq s'$, $1\leq f_1(i)\leq g(i)$  ($1\leq i\leq s_1$), $1\leq f_2(i)\leq g(s_1+i)$, ($1\leq i\leq s_2$) and $1\leq g(i)$ ($s_1+s_2+1\leq i\leq s'$) because of $\mathbf{S}\leq\mathbf{T}$

On the semantic side a new lemma, Lemma \ref{lemma_stepwise_ext0}, is necessary. The corresponding new syntactic lemma is Lemma \ref{lemma_spec_ext_case0}. To prove Lemma \ref{lemma_spec_ext_case0} axiom (EC3) has to be strengthened.

The semantic lemmas Lemma  \ref{lemma_stepwise_ext1}  and Lemma  \ref{lemma_stepwise_ext2} as well as the syntactic lemmas Lemma \ref{lemma_spec_ext_case1} (Lemma 5.6 in the published version) and Lemma \ref{lemma_spec_ext_case2} (Lemma 5.7) have to be adapted.
\begin{itemize}
\item \textit{Lemma  \ref{lemma_stepwise_ext1}} (\cite[Lemma 5.4]{Ad} in the published paper): In the published paper the chain of extensions starts with a subalgebra $\mathbf{S}$ of $\mathbf{T}=\mathbf{2}^{r}\times\prod_{i=1}^q\widehat{\mathbf{F}_{f(i)}}$ isomorphic to $\mathbf{S}\cong\widehat{\mathbf{F}_{\ell}}$, $0\leq\ell\leq\max\{f(i)\colon 1\leq i\leq q\}$. Now we have $\mathbf{S}\cong\mathbf{2}^{r}\times\prod_{i=1}^p\widehat{\mathbf{F}_{f(i)}}$, $p<q$.
\item \textit{Lemma  \ref{lemma_stepwise_ext2}} (corresponding to \cite[Lemma 5.5]{Ad}): In the published paper the chain of extensions starts with a subalgebra $\mathbf{S}$ of $\mathbf{T}=\mathbf{2}^{r'}\times\prod_{i=1}^q\widehat{\mathbf{F}_{f(i)}}$ satisfying $\mathbf{S}\cong\prod_{i=1}^q\widehat{\mathbf{F}_{f(i)}}$. The chain of extensions $\mathbf{T}_0,\ldots,\mathbf{T}_r$ can canonically be obtained by adjoining an explicit sequence of Boolean elements of $\mathbf{T}$.
Without \cite[Lemma 5.3]{Ad}  the chain to consider starts with $\mathbf{S}\cong\mathbf{2}^{r}\times\prod_{i=1}^q\widehat{\mathbf{F}_{f(i)}}$, $r<r'$. The chain is obtained by splitting a Boolean atom of $\mathbf{S}$ that is not an atom of $\mathbf{T}$.
\item \textit{Lemma  \ref{lemma_spec_ext_case1}} (corresponding to \cite[Lemma 5.6]{Ad}): In the published version we have $\mathbf{S}\cong\prod_{i=1}^q\widehat{\mathbf{F}_{f(i)}}$  and  $\mathbf{T}\cong\prod_{i=1}^{q+1}\widehat{\mathbf{F}_{f(i)}}$. Now we have $\mathbf{S}\cong\mathbf{2}^{p}\times\prod_{i=1}^q\widehat{\mathbf{F}_{f(i)}}$ and $\mathbf{T}\cong\mathbf{2}^{p}\times \prod_{i=1}^{q+1}\widehat{\mathbf{F}_{f(i)}}$.
\item \textit{Lemma  \ref{lemma_spec_ext_case2}} (corresponding to \cite[Lemma 5.7]{Ad}): In the published version we have $\mathbf{S}\cong\prod_{i=1}^q\widehat{\mathbf{F}_{f(i)}}$  and  $\mathbf{T}\cong\prod_{i=1}^{q-1}\widehat{\mathbf{F}_{f(i)}}\times\widehat{\mathbf{F}_{f(q)+1}}$. Now we have $\mathbf{S}\cong\mathbf{2}^{p}\times\prod_{i=1}^q\widehat{\mathbf{F}_{f(i)}}$ and $\mathbf{T}\cong\mathbf{2}^{p}\times \prod_{i=1}^{q-1}\widehat{\mathbf{F}_{f(i)}}\times\widehat{\mathbf{F}_{f(q)+1}}$.
\end{itemize}

The syntactic Lemma \ref{lemma_spec_ext_case3} (corresponding to \cite[Lemma 5.8]{Ad})  remains unchanged.

\section{Introduction}\label{sec_introduction}

Given a first-order theory $T$ a model companion of $T$ is an extension $T^*$ such that under very general assumptions on $T$ the class of first-order structures $\Mod(T^*)$ satisfying $T^*$ consists exactly of the existentially closed models of $T$. In this case we use the notion of model companion of a theory $T$ also to denote the class of the existentially closed models of $T$, and we speak of {\em the} model companion of $T$.

As $\mathcal{PCSL}$ consists of the models of a theory $\Sigma$ satisfying these assumptions determining a model companion $\Sigma^*$ amounts to axiomatizing $\mathcal{PCSL}^{ec}$.

Our work is based on the finite axiomatization of $\mathcal{PCSL}^{ac}$ in \cite{RAS}. We extend the axiomatization given there by five axioms to obtain a finite axiomatization of the subclass $\mathcal{PCSL}^{ec}$ of $\mathcal{PCSL}^{ac}$, thus of the model companion of $\mathcal{PCSL}$.

The paper is organized as follows: Section \ref{sec_p_semilattices} provides
the basic properties and algebraic notions concerning pseudocomplemented
semilattices, p-semilattices for short, while Section \ref{sec_model_theory} presents a summary of the relevant
model-theoretic concepts.

In Section \ref{sec_ac_pcs} we consider algebraically closed p-semilattices. We present the semantic characterization of the class $\mathcal{PCSL}^{ac}$ that is the basis of its finite axiomatization. The four axioms \ref{AC1}--\ref{AC4}, which together with the identities \eqref{equation_idempotent}--\eqref{equation_least_element}, \eqref{equation_equiv1}--\eqref{equation_equiv3} charac\-te\-ri\-ze $\mathcal{PCSL}^{ac}$, are listed.

Finally, in Section \ref{sec_finite_ax_ec} we tackle the proof of this paper's title. Before showing that axioms \ref{EC1}--\ref{EC5} are sufficient in the proof of the crucial result ---Theorem \ref{theorem_sufficiency}--- existential closedness of a p-semilattice is reduced to the extendability of subalgebras that are finite subdirectly irreducible p-semilattices to finite direct products of such p-semilattices. The necessary lemmas to deal with the occurring cases are proved beforehand.

\section{Pseudo\-com\-ple\-men\-ted semi\-lattices}\label{sec_p_semilattices}
A meet-semilattice with 0 is an algebra $\langle P;\wedge,0\rangle$ axiomatized by the identities
%
\begin{align}\label{equation_idempotent}
x\w x&=x,\\
\label{equation_commutative}
x\w y&=y\w x,\\
\label{equation_associative}
(x\w y)\w z&=x\w(y\w z),\\
\label{equation_least_element}
0\w x&=0.
\end{align}

A {\em p-semilattice} $\langle P;\wedge,^*,0\rangle$ is a meet-semilattice with 0 with an additional unary operation $^*$ that satisfies the equivalence
\begin{equation}\label{equivalence_pseudocomplement}
x\w y=0 \longleftrightarrow x\w y^*=x.
\end{equation}
Defining $x\leq y$ if $x\w y=x$ it follows from \eqref{equation_idempotent}--\eqref{equation_least_element} that $\langle P;\leq\rangle$ is a partial order with least element 0 and $x\w y=\inf\{x,y\}$. Furthermore, \eqref{equivalence_pseudocomplement} amounts to $y^*$ being the greatest element disjoint from $y$, where two elements are called disjoint if their meet is 0.
%
From \eqref{equation_idempotent}--\eqref{equivalence_pseudocomplement} we immediately obtain the very useful properties
\begin{align}
\label{implication_decreasing}
x\leq y &\implies y^*\leq x^*,\\
\label{inequality_double_complement}
x&\leq x^{**},\\
\label{equation_three_stars}
x^*&= x^{***},\\
\label{equation_two_stars}
(x\w y)^{**}&= x^{**}\w y^{**}.
\end{align}
Obviously, $1:=0^*$ is the greatest element of $P$. We define $x\parallel y$ to hold if neither $x\leq y$ nor $y\leq x$ holds. A minimal element of $P$ different from 0 is called an {\it atom}, a maximal element different from 1 is called an {\it anti-atom}. An element $d$ of $P$ satisfying $d^*=0$ is called {\it dense}, and if additionally $d\not=1$ holds, then $d$ is called a {\it proper dense} element. For $\mathbf{P}\in\mathcal{PCSL}$ the set $\D(\mathbf{P})$ denotes the subset of dense elements of $\mathbf{P}$ with $\langle \D(\mathbf{P});\wedge\rangle$ being a filter of $\langle P;\wedge\rangle$. An element $s$ is called {\it skeletal} if $s^{**}=s$. The subset of skeletal elements of $\mathbf{P}$ is denoted by $\Sk(\mathbf{P})$. We will write $\Sk(x)$ for $x\in\Sk(\mathbf{P})$ and $\D(d)$ for $d\in\D(\mathbf{P})$. From \eqref{equation_three_stars} follows $\Sk(\mathbf{P})=\set{ x^*: x\in P}$. In
$\Sk(\mathbf{P})$  the supremum of two elements exists with $\sup_{\Sk}\{a,b\}=(a^*\wedge b^*)^*$ for $a,b\in\Sk(\mathbf{P})$. Instead of $\sup_{\Sk}\{a,b\}$ we use the shorter $a\dot{\vee} b$, assuming $a,b\in\Sk(\mathbf{P})$, which follows from \eqref{implication_decreasing} and \eqref{inequality_double_complement}. Observe that $\langle\Sk(\mathbf{P});\w,\dot{\vee},^*,0,1\rangle$ is a Boolean algebra. In the subset $\Sk(\mathbf{P})$ of skeletal elements we consider the subset
$\C(\mathbf{P}):=\set {c\in\Sk(\mathbf{P}): x\geq c
\And
x\geq c^*\longra x=1\text{ for all }x\in P}$ of \textit{central} elements of $\mathbf{P}$.

From \eqref{equation_three_stars} and \eqref{equation_two_stars} we obtain
\begin{equation}\label{equation_dense_star}
\Sk(b)\And\D(d)\quad\Longrightarrow\quad(d\w b)^*=b^*,
\end{equation}
which will be used among else to show that certain sets are closed under the operation $^*$: \[(d\w b)^*=(d\w b)^{***}=((d\w b)^{**})^*=(d^{**}\w b^{**})^*=(0^*\w b^{**})^*=b^{***}=b^*\]

Equation \eqref{equation_dense_star} means that the pseudocomplement $x^*$ of a meet $x=d\w b$ of a dense and a skeletal element is again the meet of a dense and a skeletal element as $x^*=b^*=1\w b^*$.

Balbes and Horn \cite{BaHo} showed, assuming \eqref{equation_idempotent}--\eqref{equation_least_element}, that \eqref{equivalence_pseudocomplement} is equivalent to the identities
\begin{align}\label{equation_equiv1}
x\w (x\w y)^*&=x\w y^*,\\
\label{equation_equiv2}
0^*\w x&= x,\\
\label{equation_equiv3}
0^{**}&=0.
\end{align}
Thus the class $\mathcal{PCSL}$, axio\-ma\-ti\-zed by the set of iden\-ti\-ties $\Sigma:=\{\eqref{equation_idempotent},\allowbreak\eqref{equation_commutative},\allowbreak\eqref{equation_associative},\eqref{equation_least_element},\allowbreak\eqref{equation_equiv1},\eqref{equation_equiv2},\eqref{equation_equiv3}\}$, is equational. As an equational class $\mathcal{PCSL}$ is closed under products, subalgebras and homomorphisms. Therefore, every p-semilattice is a subdirect product of subdirectly irreducible p-semilattices, thus a subalgebra of a direct product of subdirectly irreducible p-se\-mi\-lat\-tices. Jones \cite{Jo} showed that $\mathcal{PCSL}$ is finitely generated by $\mathbf{3}$, the p-semilattice order-isomorphic to the three-element chain $0<e<1$. With \cite[Corollary 3.8]{Be} we obtain that the class  $\mathcal{PCSL}$ is {\em locally finite}. This fact is also proved in \cite{Jo}, where it is shown that a free p-semilattice with finitely many generators is finite (Corollary 3.1).

To characterize the subdirectly irreducible p-semilattices we define for any p-semilattice $\mathbf{P}$ the
p-semilattice $\widehat{\mathbf{P}}$ to be the p-semilattice obtained
from $\mathbf{P}$ by adding a new top element. The maximal proper dense
element of $\widehat{\mathbf{P}}$ is denoted by
$e$. Jones \cite{Jo} showed that the p-semilattices $\widehat{\mathbf{B}}$ with $\mathbf{B}$
being a Boolean algebra are exactly the subdirectly irreducible
p-semilattices. Moreover, let $\mathbf{2}$ denote the two-element Boolean
algebra, $\mathbf{F}_n$ the $n$-atom Boolean algebra and $\mathbf{A}$ the countable atomfree Boolean algebra interpreted as p-semilattices. $\mathbf{F}_0$ then is the one-element Boolean algebra and $\widehat{\mathbf{F}_0}=\mathbf{2}$.

For a p-semilattice $\mathbf{P}$ and an arbitrary element $a\in P$ the binary relation $x\theta_a y:\Longleftrightarrow a\wedge x=a\wedge y$ is a congruence. The factor algebra $\mathbf{P}/\theta_a$
is isomorphic to $\mathbf{P'}:=\langle\{a\wedge x\colon x\in P\};\cdot,',0\rangle$ where $\langle P';\wedge,0\rangle$ is the sub-meet semilattice  of $\langle P;\wedge,0\rangle$ and $'$ the associated pseudocomplementation.
Given the direct product $\prod_{i=1}^n \mathbf{P}_i$ and $a=(0,\ldots,0,1,\ldots,1)$ with the first $k$ places being 0, the factor algebra $\left(\prod_{i=1}^n \mathbf{P}_i\right)/\theta_a$ is isomorphic to $\prod_{i=k+1}^n \mathbf{P}_i$.
Furthermore, the map $\nu_a\colon P\to P/\theta_a$ defined by $\nu_{a}(x)=a\wedge x$ is a surjective homomorphism.


Finally, we need the notion of a {\it homomorphism over a set}: Let $\mathbf{P}$ and $\mathbf{Q}$ be p-semilattices, $\{a_1,\ldots,a_m\}$ a subset of $P\cap Q$. We say a homomorphism
$f\colon P\to Q$
is {\em over}
$\{a_1,\ldots,a_m\}$ if $f(a_i)=a_i$ holds for $1\leq i\leq m$. If in this situation $f$ is an isomorphism we say that $\mathbf{P}$ and $\mathbf{Q}$ are {\it isomorphic over} $\{a_1,\ldots,a_m\}$ and write $\mathbf{P}\cong_{\{a_1,\ldots,a_m\}}\mathbf{Q}$.

For more background on p-semilattices in general consult \cite{Fr} and \cite{Jo}, for the notions concerning the problem tackled in this paper consult \cite{RAS}.

\section{Model theory}\label{sec_model_theory}
For a first-order language $\mathcal{L}$ and an $\mathcal{L}$-structure $\mathbf{M}$ with universe $M$ the language $\mathcal{L}(M)$ is obtained by adding a constant symbol for every $m\in M$.
To define the notion of model companion we first have to define the notion of model completeness. An $\mathcal{L}$-theory $T$ is said to be \emph{model complete} if for every model $\mathbf{M}$ of $T$ the set of $\mathcal{L}$-sentences $T\cup\diag(\mathbf{M})$ is complete, where $\diag(\mathbf{M})$ is the set of atomic and negated atomic $\mathcal{L}(M)$-sentences that hold in $\mathbf{M}$. $T^*$ is said to be a \emph{model companion} of $T$ if (i) every model of $T^*$ is embeddable in a model of $T$ and vice versa and (ii) $T^*$ is model complete.

An $\mathcal{L}$-structure $\mathbf{M}$ is called \emph{algebraically closed} in a class of
$\mathcal{L}$-structures $\mathfrak{M}$ if $\mathbf{M}$ satisfies every positive existential
$\mathcal{L}(M)$-sentence that happens to hold in some
extension $\mathbf{M}'$ of $\mathbf{M}$ with $\mathbf{M}'\in \mathfrak{M}$. This means for a first-order language $\mathcal{L}$ without relation symbols ---as is the case for  $\mathcal{PCSL}$--- that
an $\mathcal{L}$-structure $\mathbf{M}$ is algebraically closed in $\mathfrak{M}$ if and only if every finite system of $\mathcal{L}$-equations with
coefficients from $M$ that is solvable in some $\mathbf{M}'\in\mathfrak{M}$ with $\mathbf{M}\leq\mathbf{M}'$ already has a solution in $M$. The stronger
notion of being \emph{existentially closed} differs from algebraically closed by allowing all existential $\mathcal{L}(M)$-sentences, thus allowing also
negated equations. Finally, $\mathfrak{M}^{ac}$ and $\mathfrak{M}^{ec}$ denote the subclass of algebraically and existentially closed models of $\mathfrak{M}$, respectively.

In the class of fields
existential and algebraic closedness coincide: If $\mathbf{K}$ is a field and
$p\left(\overrightarrow{x}\right)$ and $q\left(\overrightarrow{x}\right)$ are polynomials over $\mathbf{K}$, then the
satisfiability of the negated equation
$p\left(\overrightarrow{x}\right)\not=q\left(\overrightarrow{x}\right)$ is equivalent to the satisfiability of the
equation $x\cdot\left(p\left(\overrightarrow{x}\right)-q\left(\overrightarrow{x}\right)\right)=1$ assuming $x$ is not among the variables $\ora{x}$. Thus every system of
negated equations over $\mathbf{K}$ can be replaced by a system of equations.

However, the following examples show that this is not the general situation: In the class of Boolean algebras every Boolean algebra is algebraically closed whereas a Boolean algebra $\mathbf{B}$ is existentially
closed if and only if $\mathbf{B}$  is atomfree. An abelian group $\mathbf{G}$ is algebraically closed if and only if $\mathbf{G}$ is divisible, whereas $\mathbf{G}$  is existentially closed if and only if $\mathbf{G}$  is divisible and contains an
infinite direct sum of copies of $\Q/\Z$ (as a module). For a more detailed description of the notion of algebraic and existential closedness we refer the reader to \cite{Mac}.

There is the following close relationship between a model companion $T^*$ of $T$ and the class of its existentially closed models $\Mod(T)^{ec}$. If $T$ is \emph{inductive} ---that is, $\Mod(T)$ is closed under the union of chains--- then we have $\Mod(T^*)=\Mod(T)^{ec}$. Thus any axiomatization of the existentially closed models of $T$ is a model companion of $T$ if $T$ is inductive.

As $\mathcal{PCSL}$ is a finitely generated universal Horn class with both the joint embedding and the amalgamation property, $\mathcal{PCSL}$ has a model companion. The joint embedding property is \cite[Theorem 6.1]{Jo},  the amalgamation property is \cite[Theorem 9.1]{Jo}. The model companion need not exist with groups and commutative rings serving as examples. Because the set of identities $\Sigma$ axiomatizing $\mathcal{PCSL}$ is inductive, we have $\Mod(\Sigma^*)=\mathcal{PCSL}^{ec}$.
\section{The class $\mathcal{PCSL}^{ac}$}\label{sec_ac_pcs}

On various occasions we will use the following ---semantic---
characterization of algebraically closed p-semilattices, established
in \cite{Sc3}.

\begin{theorem}\label{theorem_ac_char}
A p-semilattice $\mathbf{P}$ is algebraically closed if and only if for any finite
subalgebra $\mathbf{S}\leq \mathbf{P}$ there exist $r,s\in\N$ and a p-semilattice $\mathbf{S'}$ isomorphic
to $\mathbf{2}^r\times\left(\widehat{\mathbf{A}}\right)^s$ such that
$\mathbf{S}\leq \mathbf{S'}\leq \mathbf{P}$.
\end{theorem}
In \cite{RAS} the list of axioms below is introduced to
axiomatize the class of algebraically closed p-semilattices. These axioms as well as the axioms \ref{EC1}--\ref{EC5} introduced in Section \ref{sec_finite_ax_ec} to axiomatize existential closedness are $\forall\exists$-sentences. The $\forall$-quantified variables represent constants in a p-semilattice $\mathbf{P}$, whereas the $\exists$-quantified variables represent elements that exist in an extension $\mathbf{Q}$ and so must exist in $\mathbf{P}$ if $\mathbf{P}$ is existentially closed. Each of these two types of variables can represent either an arbitrary element, a skeletal element or a dense element. Therefore we distinguish six types of variables. To identify the variables easily within these axioms we adopt the following conventions:
\begin{itemize}
\item $a,a_1,a_2,\ldots$ for $\forall$-quantified arbitrary constants,
\item $b,b_1,b_2,\ldots$ for $\forall$-quantified skeletal (Boolean) constants,
\item $d,d_1,d_2,\ldots$ for $\forall$-quantified dense constants,
\item $x,x_1,x_2,\ldots$ for $\exists$-quantified arbitrary elements,
\item $y,y_1,y_2,\ldots$ for $\exists$-quantified skeletal (Boolean) elements,
\item $z,z_1,z_2,\ldots$ for $\exists$-quantified dense elements.
\end{itemize}

\begin{definition}\label{def_finax_ac}
Let $\mathbf{P}$ be a p-semilattice. $\mathbf{P}$ will be said to satisfy 
 \setenumerate{label=(AC\theenumi)}
\begin{enumerate}
\item\label{AC1} if
\begin{multline*}
(\forall a_1,a_2,a_3)(a_3\geq a_1\w a_2\lra \\
(\exists
x_1,x_2)( x_1\geq a_1\And x_2\geq a_2\And x_1 \wedge x_2 = a_3)),
\end{multline*}
\item\label{AC2}
if
\begin{multline*}
(\forall a,d_1,d_2,d_3)(d_1<d_2<d_3\And   a\wedge d_1 < a\wedge d_2 < a \wedge d_3\lra\\
 (\exists z)(d_1 <z < d_3\And  z \wedge d_2= d_1\And  a \wedge d_1 < a \wedge z < a \wedge d_3)),
\end{multline*}
\item\label{AC3}  if
\begin{multline*}
(\forall d,d_m,b,b_1,b_2,a\in P)(d\parallel d_m\And
b_1 \leq d_m\And
    b_2 \leq d\And \\
    b_2 \not\leq d_m\And
    b \leq d \And
    b^* \w b_1 \not\leq d\And
    a^* \leq d_m\lra \\
(\exists y)(b \leq y \leq d\And   y^* \w b_1 \not\leq d\And y \w b_2 \not\leq d_m\And (y \w a)^* \leq d_m)),
\end{multline*}
\item\label{AC4} if
\[
(\forall b,d)(b< d<1\lra
(\exists  y)(b_1 < y < d\And b\bsup y^* < d)).
\]
\end{enumerate}
\end{definition}

The following theorem, the main result of \cite{RAS}, states that the preceding list of axioms together with a finite axiomatization of the class $\mathcal{PCSL}$ is a finite axiomatization of the class $\mathcal{PCSL}^{ac}$.

\begin{theorem}\label{theorem_finite ax_ac}
A p-semilattice $\mathbf{P}$ is algebraically closed if and only if
$\mathbf{P}$ satisfies the axioms \ref{AC1}--\ref{AC4}.
\end{theorem}
The following lemma will be used in Lemma \ref{lemma_criterion_ec} to put a simplifying assumption on the set of coefficients $S$ of a system of equations and negated equations that can be solved in an extension of a p-semilattice $\mathbf{P}$  whose existential closedness is to be shown.
\begin{lemma}\label{lemma_fin_dir_prod}
Let $\mathbf{P}$ be a p-semilattice satisfying \ref{AC1}--\ref{AC4}. Then for every finite
subalgebra $\mathbf{S}\leq \mathbf{P}$ there exists $\mathbf{S}'\leq \mathbf{P}$ such that $\mathbf{S}\leq \mathbf{S'}$ and $\mathbf{S}'\cong\mathbf{2}^k\times\prod_{i=1}^\ell\widehat{\mathbf{F}_{h(i)}}$ with $k,\ell\in\N$, $h\colon\{1,\ldots,\ell\}\rightarrow\N\setminus\{0\}$.
\end{lemma}
\begin{proof}
The statement of the lemma is an intermediate result in the proof of \cite[Theorem 9.1]{RAS}.
\end{proof}
%
%
\section{A finite axiomatization of $\mathcal{PCSL}^{ec}$}\label{sec_finite_ax_ec}
%
The following criterion for a p-semilattice to be existentially closed will determine this section.
\begin{lemma}\label{lemma_criterion_ec}
A  p-semilattice $\mathbf{P}$ satisfying axioms \ref{AC1}--\ref{AC4} is existentially closed if 
\setenumerate{label=(\theenumi)}
\begin{enumerate}
\item\label{lemma_criterion_ec_item1} $\mathbf{S}\leq\mathbf{P}\leq\mathbf{Q}$, $\mathbf{S}\leq\mathbf{T}\leq\mathbf{Q}$
\item\label{lemma_criterion_ec_item2} 
\begin{equation}\label{eq_apprex}
      \mathbf{S}\cong\mathbf{2}^r\times\prod_{i=1}^{s_1}\mathbf{F}_{f_1(i)}\times\prod_{i=1}^{s_2}\widehat{\mathbf{F}_{f_2(i)}}\quad (r,s_1,s_2\in\N)
\end{equation}
%
\begin{equation}\label{eq_v1_vm}
\mathbf{T}\cong\mathbf{2}^{r'}\times\prod_{i=1}^{s'}\widehat{\mathbf{F}_{g(i)}} \quad(r',s', g(i)\in\N),
\end{equation}
where we have $r\leq r'$, $s_1+s_2\leq s'$, $1\leq f_1(i)\leq g(i)$  ($1\leq i\leq s_1$), $1\leq f_2(i)\leq g(s_1+i)$, ($1\leq i\leq s_2$) and $1\leq g(i)$ ($s_1+s_2+1\leq i\leq s'$) because of $\mathbf{S}\leq\mathbf{T}$,
\end{enumerate}
imply the existence of a subalgebra $\mathbf{S'}$ of $\mathbf{P}$ with $\mathbf{S}\leq\mathbf{S'}\cong_S\mathbf{T}$.
\end{lemma}
\begin{proof}
First we observe that $\mathbf{P}$ is existentially
closed if for any extension $\mathbf{Q}$ of
$\mathbf{P}$ with $a_1,\ldots,a_m\in P$ and $v_1,\ldots,v_n\in Q$
arbitrary, there exist $u_1,\ldots,u_n\in P$ such that $\Sg{P}(\{
a_1,\ldots,a_m,u_1,\ldots,u_n\})$ and $\Sg{Q}(\{
a_1,\ldots,a_m,v_1,\ldots,v_n\})$ are isomorphic over $\{
a_1,\ldots,a_m\}$:

Every finite system of equations and negated equations with coefficients
$a_1,\ldots,a_m\in P$ corresponds to a formula
$\varphi(\ora{x},\ora{a})$, with
$\varphi$ being a quantifier-free $\mathcal{L}(\mathbf{P})$-formula. If $\mathbf{Q}\models(\exists
\ora{x})\varphi(\ora{x},\ora{a})$, say
$\mathbf{Q}\models\varphi(\overrightarrow{w},\overrightarrow{a})$,
then there exist $r_1,\ldots,r_n\in P$ such that by assuming the assumption of the preceding paragraph $\Sg{P}(\{
a_1,\ldots,a_m,r_1,\ldots,r_n\})$ and $\Sg{Q}(\{
a_1,\ldots,a_m,\allowbreak w_1,\ldots,w_n\})$ are isomorphic over $\{
a_1,\ldots,a_m\}$. We obtain
$\mathbf{P}\models\varphi(\overrightarrow{r},\overrightarrow{a})$,
thus $\mathbf{P}\models(\exists \ora{x})\varphi(\ora{x},\ora{a})$.

To simplify notation we define $S=\{a_1,\ldots,a_m\}$ 
and $T=\{a_1,\allowbreak\ldots,\allowbreak a_m,v_1,\allowbreak\ldots,v_n\}$, where we may assume that  $S$ and $T$ are the carrier sets of subalgebras $\mathbf{S}$ and $\mathbf{T}$ of $\mathbf{P}$ and $\mathbf{Q}$, respectively (otherwise consider $\Sg{P}(S)$ and $\Sg{Q}(T)$ and add the equations $x_i=a_i$, $i=k+1,\ldots,m$, if 
$\Sg{P}(S)\setminus S=\{a_{k+1},\ldots,a_m\}$). By local finiteness of $\mathcal{PCSL}$ we may assume $\mathbf{S}$ and $\mathbf{T}$ to be finite. 

It remains to show that we can assume $\mathbf{S}$ and $\mathbf{T}$ to satisfy  \eqref{eq_apprex} and \eqref{eq_v1_vm}, respectively. We first consider $\mathbf{S}$. With Lemma \ref{lemma_fin_dir_prod} it follows that $\mathbf{S}$ can be extended in $\mathbf{P}$ to a subalgebra as in \eqref{eq_apprex}. 

Now we turn to $\mathbf{T}$. Using subdirect representation, $\mathbf{Q}=\widehat{\mathbf{B}}^I$ may be assumed for a suitable atomfree Boolean algebra $\mathbf{B}$ and a suitable index set $I$. $\widehat{\mathbf{B}}$ is algebraically closed by Theorem \ref{theorem_ac_char}, therefore $\mathbf{Q}$ as a direct product of algebraically closed factors  is algebraically closed according to \cite[Lemma 5]{Sc3}. According to Theorem \ref{theorem_finite ax_ac}, $\mathbf{Q}$ satisfies \ref{AC1}--\ref{AC4}. According to Lemma \ref{lemma_fin_dir_prod} the finite subalgebra $\mathbf{T}$ can be extended within  $\mathbf{Q}$ to a subalgebra as in \eqref{eq_v1_vm}.
\end{proof}
%
%
According to Lemma \ref{lemma_criterion_ec} a finite axiomatization of $\mathcal{PCSL}^{ec}$ can be obtained by adding to the axioms \ref{AC1}--\ref{AC4}, which axiomatize $\mathcal{PCSL}^{ac}$, finitely many axioms $\varphi_1,\ldots,\varphi_m$ such that $\mathbf{P}\models\bigwedge_{i=1}^m\varphi_i$ implies the existence of $\mathbf{S'}\leq\mathbf{P}$ with $\mathbf{S'}\cong_S\mathbf{T}$ whenever $\mathbf{S}$, $\mathbf{T}$ and $\mathbf{Q}$ satisfy \ref{lemma_criterion_ec_item1} and \ref{lemma_criterion_ec_item2} of Lemma \ref{lemma_criterion_ec}.

We are going to obtain such a finite axiomatization by carrying out the following steps.
\begin{enumerate}
\item\label{item_semantic} 
Given $\mathbf{S}$, $\mathbf{T}$ and $\mathbf{Q}$ satisfying \ref{lemma_criterion_ec_item1} and \ref{lemma_criterion_ec_item2} of Lemma \ref{lemma_criterion_ec} we show that there is a chain $\left(\mathbf{T}_i\right)_{0\leq i\leq n}$ of subalgebras of $\mathbf{Q}$ of type \eqref{eq_v1_vm} such that $\mathbf{T}_0=\mathbf{S}$, $\mathbf{T}_n=\mathbf{T}$ and $\mathbf{T}_i\leq\mathbf{T}_{i+1}$ for $i=0,\ldots,n-1$. 

Furthermore, for $0\leq i\leq n-1$ there is $b\in\Sk(\mathbf{T}_i)$ and $d\in\D(\mathbf{T}_i)$ such that $\mathbf{T}_{i+1}=\Sg{\mathbf{Q}}(T_{i}\cup\{b\})$ or $\mathbf{T}_i=\Sg{\mathbf{Q}}(T_{i}\cup\{d\})$, respectively.

The construction of such a chain is the contents of the Lemma \ref{lemma_stepwise_ext0}--Lemma \ref{lemma_stepwise_ext2}.  
\item\label{item_syntactic} 
We introduce the five axioms \ref{EC1}--\ref{EC5} such that if $\mathbf{P}$ sa\-tis\-fies these axioms there is a chain $\left(\mathbf{S}_i\right)_{0\leq i\leq n}$ of subalgebras of $\mathbf{P}$ such that $\mathbf{S}_i\cong_S\mathbf{T}_i$ for $i=1,\ldots,n$. 

The axioms guarantee that if the universally bound variables range over the elements of $S_i$ of the appropriate type there is an element $x\in P$ with $\Sg{\mathbf{P}}(S_{i}\cup\{x\})\cong_S\mathbf{T}_{i+1}$.

That \ref{EC1}--\ref{EC5} ensure these extension steps is the contents of the Lemmas \ref{lemma_spec_ext_case0}--\ref{lemma_spec_ext_case3}. 
\end{enumerate}
%
%
{\em Carrying out Step \ref{item_semantic}} 

In the following a direct product $\prod_{i=\ell}^k\widehat{\mathbf{B}_i}$ of subdirectly irreducible p-semilattices with $k<\ell$ is assumed to be the one-element p-semilattice.
\begin{lemma}\label{lemma_stepwise_ext0}
If $\mathbf{T}=\mathbf{2}^{r}\times\prod_{i=1}^q\widehat{\mathbf{F}_{g(i)}}$ with $1\leq g(i)$ for  $1\leq i\leq q$,
and $\mathbf{S}\leq \mathbf{T}$ such that
$\mathbf{S}\cong\mathbf{2}^{r}\times\prod_{i=1}^p\mathbf{F}_{g(i)}\times\prod_{i=p+1}^q\widehat{\mathbf{F}_{g(i)}}$, $p\leq q$, then there is a sequence of subalgebras
$\mathbf{T}_0,\ldots,\mathbf{T}_{p}$ of $\mathbf{T}$ satisfying
\begin{itemize}
\item $\mathbf{T}_0= \mathbf{S}$,
\item $\mathbf{T}_k\leq \mathbf{T}_{k+1}$ for $k=0,\ldots,p$,
\item $\mathbf{T}_k\cong \mathbf{2}^{r}\times\prod_{i=1}^{k}\widehat{\mathbf{F}_{g(i)}}\times\prod_{i=k+1}^p\mathbf{F}_{g(i)}\times\prod_{i=p+1}^q\widehat{\mathbf{F}_{g(i)}}$, $0\leq k\leq p$,
\item $\mathbf{T}_{p}= \mathbf{T}$.
\end{itemize}
\end{lemma}
\begin{proof}
To simplify notation we define $d_\ell=(1,\ldots,1\allowbreak,e,1,\ldots,1)$ where $e$ is at the $\ell$-th place, $1\leq \ell\leq r+q$. We put $\mathbf{T}_0=\mathbf{S}$ and
$\mathbf{T}_{k+1}:=\Sg{\mathbf{T}}(T_{k}\cup\{d_{r+k+1}\})$ for $0\leq k\leq p-1$. Obviously, the sequence $\left(\mathbf{T}_{k+1}\right)$ fulfils the claim of the lemma.
\end{proof}
\begin{lemma}\label{lemma_stepwise_ext1}
If $\mathbf{T}=\mathbf{2}^{r}\times\prod_{i=1}^q\widehat{\mathbf{F}_{f(i)}}$ with $1\leq f(i)$ for  $1\leq i\leq q$
and $\mathbf{S}\leq \mathbf{T}$ such that
$\mathbf{S}\cong\mathbf{2}^{r}\times\prod_{i=1}^p\widehat{\mathbf{F}_{f(i)}}$, $0\leq p<q$, then there is $g\colon\{p+1,\ldots,q\}\ra\N$ and a sequence of subalgebras
$\mathbf{T}_0,\ldots,\mathbf{T}_{2(q-p)}$ of $\mathbf{T}$ satisfying
\begin{itemize}
\item $\mathbf{T}_0= \mathbf{S}$,
\item $\mathbf{T}_k\leq \mathbf{T}_{k+1}$ for $k=0,\ldots,2(q-p)-1$,
\item $\mathbf{T}_k\cong \mathbf{2}^{r}\times\prod_{i=1}^{p} \widehat{\mathbf{F}_{f(i)}}\times\prod_{i=p+1}^{p+k} \widehat{\mathbf{F}_{g(i)}}$, $1\leq k\leq q-p$, $g(i)\leq f(i)$ for $p+1\leq i\leq q$,
\item $\mathbf{T}_{q-p+k}\cong \mathbf{2}^{r}\times\prod_{i=1}^{p+k}\widehat{\mathbf{F}_{f(i)}}\times\prod_{i=p+k+1}^q \widehat{\mathbf{F}_{g(i)}}$ for $1\leq k\leq q-p$,
\item $\mathbf{T}_{2(q-p)}= \mathbf{T}$.
\end{itemize}
\end{lemma}
\begin{proof}
As in the proof of Lemma \ref{lemma_stepwise_ext0} we first simplify notation.
We define $c_k=(0,\ldots,0,1,0,\ldots,0)$ where $1$is at the $k$-th place, $1\leq k\leq r+q$. Furthermore, we define $d_a\in\D(\mathbf{T})$ for $a\subseteq\{1,\ldots,r+q\}$ to be the dense element satisfying $\pi_i(d_a)=e$ if and only if $i\in a$. Analogously, $c_a\in\Sk(\mathbf{T})$ is defined to be the central element satisfying $\pi_i(c_a)=1$ if and only if $i\in a$.

Since $\mathbf{S}$ is the product of subdirectly irreducible factors there are $j_1,\ldots,j_p\in\{1,\ldots,r+q\}$ such that $\pi_{j_i}(S)=\pi_{j_i}(T)$ for $i=1,\ldots,r+p$. We may assume $\{j_1\ldots,j_p\}=\{r+1,\ldots,r+p\}$.

Now we look at $\pi_{r+i}(\mathbf{S})$ for $i=p+1,\ldots,q$. Due to the subdirect irreducibility of the factors of
$\mathbf{S}$ there is for $i\in\{p+1,\ldots,q\}$ an index $j\in\{1,\ldots,p\}$ such that $|\pi_i(S)|\leq|\pi_{j}(S)|$. In case of equality  
there is $j\in\{1,\ldots,p\}$ with $\pi_{r+i}(s)=\pi_{r+j}(s)$ for all $s\in S$ ---after renaming the atoms if necessary. For $1\leq j\leq p$ let $m_j\subset\{r+p+1,\ldots,r+q\}$ be the subset of these indices and $\ol{m_j}:=m_j\cup\{r+j\}$. By the definition of $m_j$ we have $\mathbf{S}/\theta_{c_{\ol{m_j}}}\cong\pi_{r+j}(\mathbf{S})$ for $1\leq j\leq p$. We have $1\leq|\ol{m_j}|$ and $m:=\sum_{j=1}^p |m_j|\leq q$. We may assume $m_j=\{r+p+\sum_{i=1}^{j-1}|m_i|+1,\ldots,r+p+\sum_{i=1}^{j}|m_i|\}$ without loss of generality.

If $|\pi_i(S)|<|\pi_{j}(S)|$ then $\pi_i(d)=1$ for
$d\in\D(\mathbf{S})$:  There are elements $a,b\in\Sk(\mathbf{S})$ such that $a_i^*=b_i$ but $a_j^*\not=b_j$. Then at least one of $a_j\w b_j^*>0$ and $a_j^*\w b_j>0$ holds, thus either $a\w b^*=(u_1,\ldots,u_{r+q})$ or $a^*\w b=(u_1,\ldots,u_{r+q})$ such that $u_j>0$ and $u_i=0$, implying $1=u_i^*\leq d_i$. Therefore we have that $p_{r+p+i}(d)=1$ for $d\in\D(\mathbf{S})$ and $r+p+m< i\leq r+q$. We define
\begin{equation}\label{stepwise_ext}
\mathbf{S}_l=
 \begin{cases}
 \pi_{r+p+l}(\mathbf{S}), & \text{if  $\pi_{r+p+l}(d)=e$ for some $d\in\D(\mathbf{S})$}; \\
 \widehat{\pi_{r+p+l}(\mathbf {S})}, & \text{if  $\pi_{r+p+l}(d)=1$ for all $d\in\D(\mathbf{S})$.}
 \end{cases}
\end{equation}
for $l=1,\ldots,q-p$. $g(i)$ used in the statement of the lemma is such that $\widehat{F_{g(i)}}\cong\mathbf{S}_{i-p}$.

We put $\mathbf{T}_0=\mathbf{S}$. Now, we are first going to extend $\mathbf{T}_0$ successively by splitting the maximal dense elements of $\mathbf{S}/\theta_{c_{\ol{m_j}}}$, $j=1,\ldots,p$, where necessary, that is where $1\leq |m_j|$ holds. $\mathbf{S}/\theta_{c_{\ol{m_j}}}$ then yields a factor isomorphic to $\pi_{r+j}(\mathbf{S})\times\prod_{i\in m_j}\mathbf{S}_{i-(r+p)}$.

Let us consider $j=1$. $d_{r+p+1}$ is maximal dense not only with repect to $\mathbf{S}/\theta_{c_{\ol{m_1}}}$ but to $\mathbf{T}$; $d_{\{r+p+2,\ldots,r+p+|m_1|\}}$ is its complement with respect to $\mathbf{S}/\theta_{c_{m_1}}$. Therefore,
\[
 \mathbf{T}_1:=\Sg{\mathbf{T}}(T_{0}\cup\{d_{r+p+1},d_{\{r+p+2,\ldots,r+p+|m_1|\}},c_{r+p+1}\}),
\]

\[
\mathbf{T}_2:=\Sg{\mathbf{T}}(T_{1}\cup\{d_{r+p+2},d_{\{r+p+3,\ldots,r+p+|m_1|\}},c_{r+p+2}\})
\]
and generally
%
\[
\mathbf{T}_{k+1}:=\Sg{\mathbf{T}}(T_{k}\cup\{d_{r+p+k+1},d_{\{r+p+k+2,\ldots,r+p+|m_1|\}},c_{r+p+k+1}\})
\]
for $0\leq k\leq |m_1|-1$. We have to show $\mathbf{T}_1\cong\mathbf{T}_{0}\times\mathbf{S}_{1}$, $\mathbf{T}_2\cong\mathbf{T}_{1}\times\mathbf{S}_{2}$ and in general
\begin{equation}\label{equation_homomorphic_extension}
 \mathbf{T}_{k+1}\cong\mathbf{T}_{k}\times\mathbf{S}_{k+1}.
\end{equation}

Considering $1\leq j\leq p$ arbitrary, $\sum_{i=1}^{j-1}|m_i|\leq k\leq\sum_{i=1}^{j}|m_i|-1$, we define
\begin{equation}\label{eq_extension_2}
	\mathbf{T}_{k+1}:= \Sg{\mathbf{T}}\left(T_{k}\cup\{d_{r+p+k+1},d_{\{r+p+k+2,\ldots,r+p+\sum_{i=1}^{j}|m_i|\}},c_{r+p+k+1}\}\right),
\end{equation}
having to show that \eqref{eq_extension_2} satisfies \eqref{equation_homomorphic_extension}.

Secondly, we consider $\ol{m}< k\leq q-p-1$, where we define
\begin{equation}\label{eq_extension_3}
	\mathbf{T}_{k+1}=\Sg{\mathbf{T}}\left(T_{k}\cup\{d_{r+k+1},c_{r+k+1}\}\right)
\end{equation}
and for which we again have to show \eqref{equation_homomorphic_extension}.

We define $\varphi_k:\mathbf{T}_{k}\ra\prod_{i=1}^{r+p+k}\pi_i(\mathbf{T})$ by $\varphi_k(x_1,\ldots,x_{r+q})=(x_1,\ldots,x_{r+p+k})$. Obviously, $\varphi_k$ is an homomorphism. We show inductively that $\varphi_k$ is injective and
\begin{equation}\label{eq_surjectivity}
\im(\varphi_k)=\prod_{i=1}^{r+p}\pi_i(S)\times\prod_{i=1}^{k}S_i.
\end{equation}
Injectivity holds since $x,y\in T_k$, $\pi_i(x)\not=\pi_i(y)$ for $i>k$ implies $\pi_i(x)\not=\pi_i(y)$ for an $i\leq k$ by the construction of $\mathbf{T}_k$. Equation \eqref{eq_surjectivity} holds for $k=0$ as $\mathbf{T}_0=\mathbf{S}$ and $\mathbf{S}\cong\prod_{i=1}^{r+p}\pi_i(\mathbf{S})$. For the induction step we consider the two cases of the definition of $\mathbf{T}_k$.
\begin{enumerate}
\item\label{first_case}
Validity of \eqref{eq_surjectivity} for \eqref{eq_extension_2}: By \eqref{eq_extension_2} and the induction hypothesis we have
\begin{equation}\label{eq_sur11}
\varphi_{k+1}(\{c_{r+p+k+1}^*\w x\colon x\in T_{k+1}\})=\prod_{i=1}^{r+p}\pi_i(S)\times\prod_{i=1}^{k}S_i\times\{0\}.
\end{equation}
By \eqref{eq_extension_2} we have
\begin{equation}\label{eq_sur12}
\varphi_{k+1}(\{c_{r+p+k+1}\w x\colon x\in T_{k+1}\})=\underbrace{(0,\ldots,0)}_{r+p+k\text{ places}}\times S_{k+1}.
\end{equation}
From \eqref{eq_sur11} and \eqref{eq_sur12} and the construction of $\mathbf{T}_{k+1}$ in \eqref{eq_extension_2} follows the claim.
\item Analogous to the first case. 
\end{enumerate}
After $q-p$ steps we obtain the subalgebra $\mathbf{T}_{q-p}$, which is isomorphic
to $\mathbf{S}\times\prod_{l=1}^{q-p} \mathbf{S}_l$. If $|S_1|<|\widehat{F_{f(p+1)}}|$, there is
$b\in\Sk(\mathbf{T}_{q-p})$ such that $b<d_{r+p+1}$ and $b$ an anti-atom
of $\Sk(\mathbf{T}_{q-p})$ but no anti-atom of $\Sk(\mathbf{T})$. There is a skeletal
element $\bar{b}$ with $b<\bar{b}<d_{r+p+1}$ and $\left.b\bsup \bar{b}\right.^*
<d_{r+p+1}$. Setting $\mathbf{T}_{q-p,1}=\Sg{T}(
  T_{q-p}\cup\{\bar{b}\})$ we obtain
\begin{equation}\label{eq_my_contr}
    \mathbf{T}_{q-p,1}= \set{((\bar{b}\wedge s)\bsup(\left.\bar{b}\right.^*\w t))\w d: s,t\in \Sk(T_{q-p}),d\in\D(\mathbf{T}_{q-p})}
\end{equation}
using conjunctive normal form for Boolean terms, \eqref{equation_dense_star} and $\D(\Sg{T}(
  T_{q-p}\cup\{\bar{b}\}))=\D(\mathbf{T}_{q-p})$.
The right hand side of \eqref{eq_my_contr} is isomorphic to $\mathbf{S}\times\widehat{\mathbf{F}_{r_1+1}}\times\prod_{l=2}^{q-p}\mathbf{S}_l$ if $r_1\in\N$ is such that $\mathbf{S}_1\cong\widehat{\mathbf{F}_{r_1}}$. Repeating this procedure for $\mathbf{T}_{q,n}$ as long as $r_1+n<f(p+1)$
yields a subalgebra $\mathbf{T}_{q+1}$ of $\mathbf{T}$ isomorphic to
$\mathbf{S}\times\widehat{\mathbf{F}_{f(p+1)}}\times\prod_{l=2}^{q-p} \mathbf{S}_l$. Applying this procedure
to the factors $\mathbf{S}_l$ for $l=2,\ldots,q-p$ finally finishes the proof.
\end{proof}
\begin{lemma}\label{lemma_stepwise_ext2}
If $\mathbf{T}=\mathbf{2}^{r'}\times\prod_{i=1}^q\widehat{\mathbf{F}_{f(i)}}$ with
$r',q,f(i)\in\N\setminus\{0\},1\leq i\leq q$, and $\mathbf{S}\leq\mathbf{T}$ such that
$\mathbf{S}\cong\mathbf{2}^{r}\times\prod_{i=1}^q\widehat{\mathbf{F}_{f(i)}}$, $0\leq r<r'$, then there is a
sequence of subalgebras $\mathbf{T}_0,\ldots, \mathbf{T}_{r'-r}$ of $\mathbf{T}$ with the following
properties:
\begin{itemize}
\item $\mathbf{T}_0=\mathbf{S}$,
  \item $\mathbf{T}_k\leq\mathbf{T}_{k+1}$ for $k=0,\ldots,r'-r-1$,
  \item  $\mathbf{T}_k\cong \mathbf{2}^{r+k}\times \prod_{i=1}^q\widehat{\mathbf{F}_{f(i)}}$ for $k=0,\ldots,r'-r$.
\end{itemize}
\end{lemma}
\begin{proof}
The subalgebra $\mathbf{T}_{k+1}$ can obtained from $\mathbf{T}_k$ by splitting an atom of $\mathbf{T}_k$ that is not an atom of $\mathbf{T}$.
\end{proof}
%
%
{\em Carrying out Step \ref{item_syntactic}} 

We begin with the introduction of the axioms \ref{EC1}--\ref{EC5}.
\begin{definition}\label{def_finax_ec}
Let $\mathbf{P}$ be a p-semilattice. $\mathbf{P}$ will be said to satisfy
%
\setenumerate{label=(EC\theenumi)}
\begin{enumerate}
\item\label{EC1} if
\begin{equation*}(\forall b_1,b_2)(b_1<b_2\lra (\exists y)(b_1<y<b_2)),
\end{equation*}
%
\item\label{EC2} if
%
\begin{multline*}(\forall b_1,b_2,d)(b_1\leq b_2<d\And b_1^*\parallel d\lra\\
(\exists
y)(b_2<y<1\And b_1^*\w y\parallel d\And
b_1 \bsup\, y^*<d)),
\end{multline*}
\item\label{EC3} if
\begin{equation*}(\forall b)(b<1\lra (\exists z)(b<z\And z\not> b^*)). 
\end{equation*}
\item\label{EC4} if
\begin{equation*}(\forall d_1,d_2)(d_1<d_2\lra (\exists z)(d_1<z<d_2)),
\end{equation*}

\item\label{EC5} if
\[(\forall b,d)(0<b<d
\lra
(\exists
z)(z<d\And  b\parallel z\And d\w b^*=z\w b^*)),
\]
\end{enumerate}
\end{definition}

A couple of sentences to explain what the axioms \ref{EC1}--{\ref{EC5} mean are appro\-priate. Axioms \ref{EC1}  and \ref{EC4} are the usual density conditions holding in existentially
closed posets for skeletal and dense elements. Skeletal and dense elements must be mentioned separately because $b_1<y<b_2$ with $b_1$ and $b_2$ skeletal does not imply that $y$ is skeletal as well. 

In a p-semilattice $\mathbf{P}$ satisfying \ref{EC3} a finite subalgebra $\mathbf{S}\cong
\mathbf{2}^p$ with $1\leq p$ can be extended to a subalgebra $\mathbf{S'}$ isomorphic to
$\mathbf{T}$ over $S$ for any subalgebra $\mathbf{T}\cong
\mathbf{2}^p\times\widehat{\mathbf{F}_n}$ of an extension $\mathbf{Q}$ of $\mathbf{P}$ that is an extension of $\mathbf{S}$ with
$1\leq n$.

To understand \ref{EC2} and \ref{EC5} diagrams may be helpful.\\

\begin{picture}(310,65)
\put(0,0){$b_1$}
\put(10,0){$\circ$}
\put(0,30){$b_2$}
\put(10,30){$\circ$}
\put(0,60){$d$}
\put(10,60){$\circ$}
\put(12.5,4.2){\line(0,1){26.6}}
\put(12.5,60.8){\line(0,-1){26.4}}

\put(50,60){$\circ$} \put(60,60){$b_1^*$} \put(25,40){$d\parallel
b_1^*$} \put(225,15){$d\parallel b_1^*\w y$}

\put(100,30){$\stackrel{\ref{EC2}}{\Longrightarrow}$}

\put(150,0){$y^*$}\put(160,0){$\circ$} \put(180,30){$\circ$}
\put(150,30){$b_2\bsup y^*$}
\put(180,60){$\circ$}\put(182.5,34.4){\line(0,1){26.4}}
\put(184,31.5){\line(2,-3){18.4}} \put(201,0){$\circ$}
\put(204.9,3.7){\line(2,3){38.3}} \put(214,0){$b_2$}
\put(170,60){$d$}\put(181,31.3){\line(-2,-3){17.9}}

\put(263.7,30){$\circ$} \put(275,30){$b_1^*\w y$}
\put(242,60){$\circ$} \put(286,60){$\circ$}
\put(295,60){$b_1^*$}
\put(227,60){$y$}\put(268.1,32.9){\line(2,3){19}}\put(264.4,33){\line(-2,3){18.7}}

\end{picture}\\

In a p-semilattice $\mathbf{P}$ satisfying \ref{EC2} a finite subalgebra $\mathbf{S}\cong \prod_{i=1}^q\widehat{\mathbf{F}_{f(i)}}$ with $1\leq f(i)$ for $1\leq i\leq q$ can be extended to a subalgebra $\mathbf{S'}$ isomorphic to $\mathbf{T}$ over $S$ for any subalgebra $\mathbf{T}\cong \mathbf{2}\times\prod_{i=1}^q\widehat{\mathbf{F}_{f(i)}}$ of an extension $\mathbf{Q}$ of $\mathbf{P}$ that is an extension of $\mathbf{S}$. Applying \ref{EC2} to suitable $d,b_1,b_2\in S$ yields a skeletal element $y$ that behaves with respect to $\mathbf{S}$ as the element $(0,1,\ldots,1)\in T\setminus S$.\\

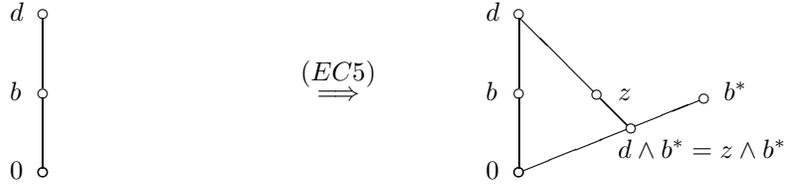
\begin{picture}(310,65)
\put(0,0){$0$}\put(10,30){$\circ$}
\put(10,0){$\circ$}\put(0,30){$b$} \put(10,0){$\circ$}
\put(0,60){$d$} \put(10,60){$\circ$} \put(12.5,4.2){\line(0,1){26.4}}
\put(12.5,60.8){\line(0,-1){26.2}}

\put(110,30){$\stackrel{\ref{EC5}}{\Longrightarrow}$}

\put(180,0){$0$}\put(190,30){$\circ$}
\put(190,0){$\circ$}\put(180,30){$b$} \put(190,0){$\circ$}
\put(180,60){$d$} \put(190,60){$\circ$}
\put(220.7,33.3){\line(-1,1){27.4}}\put(223.4,30.6){\line(1,-1){10.3}}
\put(194.1,3.2){\line(5,2){39.2}} 

\put(260,28){$\circ$} \put(230,30){$z$} \put(219.5,29.5){$\circ$}
 \put(270,30){$b^*$}
 \put(192.5,4.2){\line(0,1){26.6}} \put(192.5,60.8){\line(0,-1){26.4}}
\put(232.5,16.7){$\circ$}\put(230,8){$d\w b^*=z\w b^*$}
\put(236.5,20.2){\line(5,2){24.3}}
\end{picture}\\

In a p-semilattice $\mathbf{P}$ satisfying \ref{EC5} a finite subalgebra $\mathbf{S}\cong
\mathbf{2}^p\times\prod_{i=1}^q\widehat{\mathbf{F}_{f(i)}}$ with $0\leq p$, $1\leq q$, $1\leq f(i)$
 can be extended to a subalgebra $\mathbf{S'}$ isomorphic to
$\mathbf{T}$ over $S$ for any subalgebra $\mathbf{T}\cong
\mathbf{2}^p\times\prod_{i=1}^{q+1}\widehat{\mathbf{F}_{f(i)}}$ of an extension $\mathbf{Q}$ of $\mathbf{P}$ that is an extension of $\mathbf{S}$ with
$1\leq f(q+1)$ and $\min(\D(\mathbf{T}))<\min(\D(\mathbf{S}))$. Applying \ref{EC5} to suitable $d,b\in S$ yields a dense element $z$ that behaves with respect to $\mathbf{S}$ as the element $(e,\ldots,e)\in T\setminus S$.

\begin{remark}\label{remark_ec_ba}
\begin{enumerate}
  \item Observe in \ref{EC4} that $d^*=0\And d<d'$ implies $d'^*=0$.
  \item Let $\mathbf{P}$ be a p-semilattice satisfying \ref{EC1}. Then the subalgebra $\Sk(\mathbf{P})$ is atomfree and thus existentially closed in $\Sk(\mathbf{Q})$ for any p-semilattice $\mathbf{Q}$ extending $\mathbf{P}$.
\end{enumerate}
\end{remark}
%
%
We first show the necessity of the axioms \ref{AC1}--\ref{AC4} and \ref{EC1}--\ref{EC5} for a p-semilattice to be existentially closed.

\begin{theorem}\label{theorem_necessity}
If a p-semilattice $\mathbf{P}$ is existentially closed, then it satisfies \ref{AC1}--\ref{AC4} and \ref{EC1}--\ref{EC5}.
\end{theorem}

\begin{proof} 
We consider an arbitrary existentially closed p-semilattice $\mathbf{P}$ and show that it satisfies axioms \ref{AC1}--\ref{AC4} as well as axioms \ref{EC1}--\ref{EC5}. 

That $\mathbf{P}$ satisfies axioms \ref{AC1}--\ref{AC4} follows from Theorem
\ref{theorem_finite ax_ac} because every existentially closed
p-semilattice is algebraically closed.

To prove the necessity of the axioms \ref{EC1}--\ref{EC5} we replace in these five axioms the universally bound variables by arbitrary elements of $P$ of the appropriate type. We have to show the satisfiability in $\mathbf{P}$ of the resulting $\exists$-sentences of $\mathcal{L}(P)$.
\begin{itemize}
\item $\varphi_1(b_1,b_2)$: $(\exists x)(\Sk(x)\And  b_1<x< b_2)$ with
$\Sk(b_1),\Sk(b_2)$ and $b_1<b_2$
\item $\varphi_2(b_1,b_2,d)$: $(\exists x)(\Sk(x)\And  b_2<x<1\And b_1^*\w x\parallel d\And  b_2\bsup x^*<d)$ with $\Sk(b_1),\Sk(b_2),\D(d)$, $b_1^*\parallel d$  and $b_1\leq b_2<d<1$
\item $\varphi_3(b)$:  $(\exists x)(\D(x)\And  b<x<1\And x\not> b^*)$ with $\Sk(b)$, $b<1$ 
\item $\varphi_4(d_1,d_2)$:  $(\exists x)(d_1<x<d_2)$ with
$\D(d_1),\D(d_2)$ and $d_1<d_2$
\item $\varphi_5(b,d)$: $(\exists x)(\D(x)\And  x< d\And x\parallel b\And x\w b^*=d\w b^*)$ with $\mathsf{D}(d)\And \Sk(b)\And \allowbreak 0<b<d$
\end{itemize}

To obtain the satisfiability of the sentences $\varphi_1$ and $\varphi_4$ we use that $\mathbf{P}$ as a subdirect product of subdirectly irreducible p-semilat\-ti\-ces  can be embedded in some direct product $\mathbf{Q}$ of subdirectly irreducible p-semilattices.  With suitably many factors $\mathbf{2}$ and $\widehat{\mathbf{\mathbf{B}}_i}$ each of these sentences can be satisfied in a suitable $\mathbf{Q}$. Thus they can also be satisfied in $\mathbf{P}$ if $\mathbf{P}$ is existentially closed.

For $\varphi_2$ let $U$ be an ultrafilter of $\Sk(\mathbf{P})$ not containing $b_2$. Such an $U$ exists since $b_2<1$.   We define $f\colon\mathbf{P}\to\mathbf{P}\times\mathbf{2}$ by
\begin{equation}\label{eq_nec_ec3}
f(x)=\begin{cases} (x,1) & x^{**}\in U,\\
(x,0) & \text{otherwise}.\end{cases}
\end{equation}
As $\mathbf{P}$ satisfies \ref{AC1} there is for every $x\in P$ a dense element $d_x\in P$ such that $x=d_x\w x^{**}$ as is shown in \cite[Lemma 4.5]{RAS}. Therefore, $f$ is a homomorphism. We have $f(b_1)=(b_1,0)$, $f(b_2)=(b_2,0)$, $f(d)=(d,1)$. The extension $\mathbf{P}\times\mathbf{2}$ contains $d=(1,0)$ satisfying $\varphi_2$. 

For $\varphi_3$ let $U$ be an ultrafilter of $\Sk(\mathbf{P})$ not containing $b$.  We define $f\colon\mathbf{P}\to\mathbf{P}\times\mathbf{3}$ as in \eqref{eq_nec_ec3}.
We have $f(b)=(b,0)$. The extension $\mathbf{P}\times\mathbf{3}$ contains $d=(1,e)$ satisfying $\varphi_3$. 

For $\varphi_5$ consider an ultrafilter $U$ containing $b$ and define $f$ as in \eqref{eq_nec_ec3}. The extension $\mathbf{P}\times\mathbf{3}$ contains $d=(1,e)$ satisfying $\varphi_5$.
\end{proof}


%
The following lemmas can, as mentioned earlier, be considered the syntactic counterparts of Lemmas \ref{lemma_stepwise_ext0}, \ref{lemma_stepwise_ext1} and \ref{lemma_stepwise_ext2} (Lemma \ref{lemma_stepwise_ext2} corresponding to the two lemmas Lemma \ref{lemma_spec_ext_case1} and Lemma \ref{lemma_spec_ext_case2}).  Lemma \ref{lemma_spec_ext_case1} states that if $\mathbf{S}$ is a finite subdirectly irreducible subalgebra of a p-semilattice $\mathbf{P}$ that satisfies \ref{AC1}--\ref{AC4} and \ref{EC1}--\ref{EC5}, then $\mathbf{P}$ contains a sequence $\mathbf{S}_i$, $i=0,\ldots,q$, of subalgebras satisfying $\mbf{S}_i\cong\mbf{T}_i$ for $i=0,\ldots,q$ with $\mbf{T}_0,\ldots,\mbf{T}_q$ as in Lemma \ref{lemma_stepwise_ext1}. Lemma \ref{lemma_spec_ext_case2} is the corresponding statement for the sequence $\mbf{T}_{q+1},\ldots,\mbf{T}_{2q}$ of Lemma \ref{lemma_stepwise_ext1}, whereas Lemma \ref{lemma_spec_ext_case3} is the corresponding statement for the sequence $\mbf{T}_{0},\ldots,\mbf{T}_{p}$ of Lemma \ref{lemma_stepwise_ext2}.
\begin{lemma}\label{lemma_spec_ext_case0}
Let $\mathbf{P}$ and $\mathbf{Q}$ be p-semilattices with $\mathbf{Q}$ being an extension of $\mathbf{P}$, let $\mathbf{S}\cong\mathbf{2}^r\times\prod_{i=1}^p \mathbf{F}_{g(i)}\times\prod_{i=p+1}^q\widehat{\mathbf{F}_{g(i)}}$ be a finite subalgebra of $\mathbf{P}$ with $0<p<q$, $g(i)\geq 1$ for $1\leq i\leq q$.
Furthermore, we assume that $\mathbf{T}\cong\mathbf{2}^r\times\widehat{\mathbf{F}_{g(1)}}\times\prod_{i=2}^p \mathbf{F}_{g(i)}\times\prod_{i=p+1}^q\widehat{\mathbf{F}_{g(i)}}$ is a finite subalgebra of $\mathbf{Q}$ that is an extension of $\mathbf{S}$. If $\mathbf{P}$ satisfies \ref{AC1}--\ref{AC4} and
\ref{EC1}--\ref{EC5}, then there is an
  extension $\mathbf{S'}$ of $\mathbf{S}$ in $\mathbf{P}$ satisfying $\mathbf{S'}\cong_S\mathbf{T}$.
\end{lemma}

\begin{proof}
Since $\mathbf{T}\cong\mathbf{2}^r\times\widehat{\mathbf{F}_{g(1)}}\times\prod_{i=2}^p \mathbf{F}_{g(i)}\times\prod_{i=p+1}^q\widehat{\mathbf{F}_{g(i)}}$ we may assume
$\mathbf{T}=\mathbf{2}^r\times\widehat{\mathbf{F}_{g(1)}}\times\prod_{i=2}^p \mathbf{F}_{g(i)}\times\prod_{i=p+1}^q\widehat{\mathbf{F}_{g(i)}}$ identifying the subalgebra
$\mathbf{T}$ of $\mathbf{Q}$ with the direct product $\mathbf{T}$ is isomorphic to. There is a maximal dense element $d$ in $T\setminus S$ and a maximal central element $c$ with $c<d$.
We can assume $d=d_{r+1}$ and $c=c_{r+1}^*$ using the notation of the proofs of Lemma \ref{lemma_stepwise_ext0} and Lemma \ref{lemma_stepwise_ext1}. We have $c\in S$ as $\Sk(\mathbf{T})=\Sk(\mathbf{S})$.

We then have property (0) $T=S\cup\{d\w s\colon s\in S\text{ and }s\not\leq d\}=S\cup\{d\w s\colon s\in S\text{ and }\pi_{r+1}(s)=1\}=S\cup\{d\w s\colon s\in S\text{ and }c^*\leq s\}$. Now, for $s_1,s_2\in S$ the properties (1) $c^*\leq s_2\ra s_1\not=d\w s_2$ and (2) $c^*\leq s_1\w s_2,s_1\not=s_2\ra d\w s_1\not=d\w s_2$ obviously hold.

Applying \ref{EC3} yields a dense element $\tilde{d}\in P$ for which $c<\tilde{d}$ and $ \tilde{d}\not\geq c^*$ holds. For $\mathbf{S'}:=\Sg{\mathbf{P}}(S\cup\{\tilde{d}\})$ we have $\mathbf{S'}\cong_S\mathbf{T}$. This is obtained by showing that the homomorphism $f\colon T\to S'$ defined by $f(s)$ for $s\in S$ and $f(d)=\tilde{d}$ is an isomorphism. The map $f$ is surjective due to the above property (0) and the corresponding $S'=S\cup\{\tilde{d}\w s\colon s\in S\text{ and }s\not\leq \tilde{d}\}=S\cup\{\tilde{d}\w s\colon s\in S\text{ and }\pi_{r+1}(s)=1\}=S\cup\{\tilde{d}\w s\colon s\in S\text{ and }c^*\leq s\}$. It is injective due to the above properties (1) and (2) and the corresponding properties (1') $s_1,s_2\in S$, $c^*\leq s_2\ra s_1\not=\tilde{d}\w s_2$, (2') $s_1,s_2\in S$, $c^*\leq s_1\w s_2,s_1\not=s_2\ra \tilde{d}\w s_1\not=\tilde{d}\w s_2$ for $\mathbf{S'}$.

(1') and (2') are obtained as follows: We first deal with (1'). We show that the assumption $c^*\leq s_2$ together with $s_1=\tilde{d}\w s_2$ leads to a contradiction. We obtain the sequence of implications $s_1=\tilde{d}\w s_2\Longrightarrow c^*\w s_1=c^*\w\tilde{d}\Longrightarrow (c^*\w s_1)^{**}=(c^*\w\tilde{d})^{**}\Longrightarrow c^*\w s_1^{**}=c^*\Longrightarrow c^*\leq s_1^{**}$. The last inequality implies $\pi_{r+1}(s_1^{**})=1$, which implies $\pi_{r+1}(s_1)=1$ as there is no element $x\in S$ with $\pi_{r+1}(x)<(\pi_{r+1}(x))^{**}$. Thus $c^*\leq s_1$ implying $c^*\w s_1=c^*$. Due to $\tilde{d}\not> c^*$ we have $c^*\w(\tilde{d}\w s_2)<c^*$. Thus $s_1\not=\tilde{d}\w s_2$ contradicting $s_1=\widetilde{d}\w s_2$. For (2') we first note that the assumption $c^*\leq s_1\w s_2$ means $\pi_{r+1}(s_1)=\pi_{r+1}(s_2)=1$, which is the same as $c^*\w s_1=c^*\w s_2$. This yields $c\w s_1\not=c\w s_2$, which implies $\tilde{d}\w s_1\not=\tilde{d}\w s_2$ as $\tilde{d}>c$.
\end{proof}
\begin{lemma}\label{lemma_spec_ext_case1}
Let $\mathbf{P}$ and $\mathbf{Q}$ be p-semilattices with $\mathbf{Q}$ being an extension of $\mathbf{P}$, let $\mathbf{S}\cong\mathbf{2}^p\times\prod_{i=1}^q\widehat{\mathbf{F}_{f(i)}}$ be a finite subalgebra of $\mathbf{P}$ with $p\geq 0$, $f(i)\geq 1$ for $1\leq i\leq q$.
Furthermore, we assume that $\mathbf{T}\cong\mathbf{2}^p\times\prod_{i=1}^{q+1}\widehat{\mathbf{F}_{f(i)}}$ with $f(q+1)\geq 1$ 
is a finite subalgebra of $\mathbf{Q}$ that is an extension of $\mathbf{S}$. If $\mathbf{P}$ satisfies \ref{AC1}--\ref{AC4} and
\ref{EC1}--\ref{EC5}, then there is an
  extension $\mathbf{S'}$ of $\mathbf{S}$ in $\mathbf{P}$ satisfying $\mathbf{S'}\cong_S\mathbf{T}$.
\end{lemma}

\begin{proof}
Again, since $\mathbf{T}\cong\mathbf{2}^p\times\prod_{i=1}^{q+1}\widehat{\mathbf{F}_{f(i)}}$ we may assume
$\mathbf{T}=\mathbf{2}^p\times\prod_{i=1}^{q+1}\widehat{\mathbf{F}_{f(i)}}$ identifying the subalgebra
$\mathbf{T}$ of $\mathbf{Q}$ with the direct product $\mathbf{T}$ is isomorphic to.  To simplify notation we define $\ora{x}=(x_1,\ldots,x_{p+q})$ for $x\in
T$, $\overrightarrow{x}\leq \overrightarrow{y}$ if $x,y\in T$ and
$x_i\leq y_i$ for $1\leq i\leq p+q$, and
$\overrightarrow{x}<\overrightarrow{y}$ if $\overrightarrow{x}\leq
\overrightarrow{y}$ and $x_k<y_k$ for a $k\in\{1,\ldots,p+q\}$.
Furthermore, we set
$\overrightarrow{U}=\set{\overrightarrow{x}: x\in U}$ if $U$ is a
subset of $T$.

Again, since $\mathbf{S}$ is isomorphic to the direct product of the
subdirectly irreducible factors $\mathbf{2}$ and $\widehat{\mathbf{F}_{f(i)}}$ for
$i=1,\ldots,p+q$, and since
$\mathbf{T}=\mathbf{2}^p\times\prod_{i=1}^{q+1}\widehat{\mathbf{F}_{f(i)}}$ is an
extension of $\mathbf{S}$ we have ---changing the enumeration if
necessary--- $\ora{S}=\ora{T}$, which implies $\pi_i(S)=\pi_i(T)$ for
$i=1,\ldots,p+q$. 
We set $d_0=\min(\D(\mathbf{T}))=(1,\ldots,1,e,\ldots,e)$ and consider two cases:

(1) $\min(\pi_{p+q+1}(\D(\mathbf{S})))=e$, that is $\min(\D(\mathbf{S}))=\min(\D(\mathbf{T}))$

(2) $\min(\pi_{p+q+1}(\D(\mathbf{S})))=1$

We will in both cases first attend to the dense elements. We will extend $S$ with a dense element $d$ by applying \ref{EC4} and \ref{EC5}, respectively such that firstly $\mathbf{S}_1:=\Sg{P}(S\cup\{d\})$ can be embedded over $S$ into $\mathbf{T}$ and secondly the application of \ref{AC1}--\ref{AC4} to $\mathbf{S}_1$ yields a subalgebra $\mathbf{S}_2$ such that $\D(\mathbf{S}_2)\cong_S\D(\mathbf{T})$ . Once more applying \ref{AC3} and \ref{AC4} to $\Sg{P}(S\cup\{\D(\mathbf{S}_2)\})$ will finally yield the desired subalgebra $\mathbf{S'}$.

(1): There is a $k\in\{1,\ldots,p+q\}$ such that
$\pi_k(\mathbf{S})\cong\pi_{p+q+1}(\mathbf{S})$ and $\pi_k(x)=\pi_{p+q+1}(x)$ (after renaming the atoms of $\pi_{p+q+1}(\mathbf{S})$ if necessary) for
$x\in S$: $|\pi_k(\mathbf{S})|>|\pi_{p+q+1}(\mathbf{S})|$ for all $k\in\{1,\ldots,p+q\}$ would contradict $\mathbf{S}$ being the direct product of subdirectly
irreducible factors as we assume
$\mathbf{S}\cong\mathbf{2}^p\times\prod_{i=1}^{q}\widehat{\mathbf{F}_{f(i)}}$. For $a>b$ there is no embedding of $\widehat{\mbf{F}_a}$ into $\widehat{\mbf{F}_a}\times\widehat{\mbf{F}_b}$ such that the proper dense element of $\widehat{\mbf{F}_a}$ is mapped on $(e,e)\in\widehat{\mbf{F}_a}\times\widehat{\mbf{F}_b}$, which extends to more than two factors.

There is a unique $d\in\D(\mathbf{S})$ being an anti-atom of $\mathbf{S}$ but no anti-atom of $\mathbf{T}$, thus $d=(1,\ldots,1,e,e)$ if we assume $k=p+q$. Applying axiom \ref{EC4} to $d$ and 1 yields a dense element $d_1$ such that $d<d_1<1$. Observe that for all anti-atoms $d'$ of $\mathbf{S}$ with $d'\not=d$ we have $d'\parallel d_1$ since $d'<d_1$ together with $d<d_1$ would imply $d_1=1$. There is a dense element $\wt{d_1}\in T$ such that $d<\wt{d_1}<1$. If we define $\mathbf{S}_1=\Sg{P}(S\cup\{d_1\})$ then the map $h_1\colon S_1\to T$ defined by
    \begin{equation*}\label{eq_def_h2}
    h_1(s)=
 \begin{cases}
 s, &\text{for $s\in S$},\\
 \tilde{d_1},  &\text{for $s= d_1$}
 \end{cases}
        \end{equation*}
is an embedding over $S$.

To extend $\D(\mathbf{S}_1)$ in $\mathbf{P}$ appropriately we exploit that $\mathbf{P}$ satisfies \ref{AC1}--\ref{AC4}. $\mathbf{S}_1$ can be extended in $\mathbf{P}$ to a subalgebra $\mathbf{S}_2\cong \mathbf{T}$. 

In the construction of $\mathbf{S}_2$ from $\mathbf{S}_1$ in \cite{RAS} it is not taken care of whether $\mathbf{S}_2\cong_S\mathbf{T}$. But there is a maximal dense element $d_2\in S_2$ such that $d=d_1\w d_2$. For $\mathbf{S}_3:=\Sg{P}( S\cup\{d_1,d_2\})$ we have $\D(\mathbf{S}_3)\cong\D(\mathbf{T})$ and that there is an embedding $h_3\colon S_3\to T$ extending $h_1$.
%

(2): Let $a$ be the least element of $\mathbf{S}$ such that $a\parallel d_0$. Then $a^*\w d_0=a^*\w d_1$ where $d_1:=\min(\D(\mathbf{S}))=(e,\ldots,e,1)>d_0$. Applying axiom \ref{EC5} to $d_1$ and $a$
yields a dense element $\breve{d_0}$ such that $a\parallel \breve{d_0}$ and $a^*\w \breve{d_0}=a^*\w d_1$.
Therefore, if $\mathbf{S}_1:=\Sg{P}(S\cup\{\breve{d_0}\})$ then the map $h_1\colon S_1\to T$ defined by
    \begin{equation*}\label{eq_def_h1}
         h_1(s)=
 \begin{cases}
 s, &\text{for $s\in S$},\\
 d_0,  &\text{for $s=\breve{d_0}$}
 \end{cases}
    \end{equation*}
is an embedding over $S$. 
As $\mathbf{P}$ satisfies \ref{AC1}--\ref{AC4} $\mathbf{S}_1$ can be extended in $\mathbf{P}$ to a subalgebra $\mathbf{S}_2\cong \mathbf{T}$. There is a maximal dense element $d\in S_2\setminus S_1$. 
For $\mathbf{S}_3:=\Sg{P}( S\cup\{\breve{d_0},d\})$ we have $\D(\mathbf{S}_3)\cong\D(\mathbf{T})$ and that there is an embedding $h_3\colon S_3\to T$ extending $h_1$.

Thus in both subcases there is a subalgebra $\mathbf{S}_3$ of $\mathbf{P}$ extending $\mathbf{S}$ such that $\D(\mathbf{S}_3)\cong \mathbf{2}^{q+1}$ and an embedding $h_3\colon S_3\to T$ over
$S$.
In the first subcase there are two maximal dense elements $d_1,d_2\in\D(\mathbf{S}_3)\setminus\D(\mathbf{S})$. Again proceeding as in the proof of \cite[Proposition 6.6]{RAS} applying axiom \ref{AC3} yields elements $k_1$ and $k_2$ such that $\mathbf{S}_4:=\Sg{P}(S_3\cup\{a_1,a_2\})\cong\mathbf{S}\times\pi_{p+q+1}(\mathbf{S})$:
There one defines $a_i=k_i\bsup c_0^*$ with $c_0=(\ora{0},1,\ldots,1)\in S$, the first $p$ places being 0.

The homomorphism $h_4\colon S_4\to T$ extending $h_3$ by $h_4(a_{1}):=(1,\ldots,1,\allowbreak 0,1)\in T\setminus S$ and $h_4(a_{2}):=\allowbreak(1,\ldots,\allowbreak1,0)\in T\setminus S$ is an embedding. As $h_3$ is over $S$ so is $h_4$.

In the second subcase there is by the
construction of $\mathbf{S}_1$ a unique maximal dense element $d\in\D(\mathbf{S}_3)\setminus S$.
Again proceeding as in the proof of \cite[Proposition 6.6]{RAS} we find a skeletal element
$k_d\in P$ such that $\mathbf{S}_4:=\Sg{P}(S_3\cup\{a_d\})\cong\mathbf{S}\times\pi_{p+q+1}(\mathbf{S})$, $a_d=k_d\bsup c_0^*$.
Therefore, the homomorphism $h_4\colon S_4\to T$ extending $h_4$ by $h(k_d):=(1,\ldots,1,0)\in T\setminus S$ is an embedding. As $h_3$ is over $S$ so is $h_4$.

Finally, we come to $\mathbf{S'}$. If not $\mathbf{S}_4\cong\mathbf{T}$ we apply \ref{AC4} appropriately to obtain an extension $\mathbf{S'}$ congruent to $\mathbf{T}$ and an isomorphism $h\colon S'\to T$ extending $h_4$.
\end{proof}

\begin{lemma}\label{lemma_spec_ext_case2}
Let $\mathbf{P}$ and $\mathbf{Q}$ be p-semilattices with $\mathbf{Q}$ being an extension of $\mathbf{P}$, let $\mathbf{S}\cong\mathbf{2}^{p}\times\prod_{i=1}^q\widehat{\mathbf{F}_{f(i)}}$ be a finite subalgebra of $\mathbf{P}$ with $\D(\mathbf{S})\setminus\{1\}\not=\emptyset$, and let  $\mathbf{T}\cong\mathbf{2}^{p}\times \prod_{i=1}^{q-1}\widehat{\mathbf{F}_{f(i)}}\times\widehat{\mathbf{F}_{f(q)+1}}$
be a finite subalgebra of $\mathbf{Q}$ that is an extension of $\mathbf{S}$, $0\leq p$, $1\leq f(i)$, $1\leq i\leq q$.
If $\mathbf{P}$ satisfies \ref{AC1}--\ref{AC4} and
\ref{EC1}--\ref{EC5}, then there is an extension $\mathbf{S'}$ of $\mathbf{S}$ in $\mathbf{P}$ satisfying $\mathbf{S'}\cong_S\mathbf{T}$.
\end{lemma}
\begin{proof}
There are uniquely determined $d\in\D(\mathbf{S})\setminus\{1\}$ with $d$ being an anti-atom, and $b_1\in\Sk(\mathbf{S})$ such that
  $b_1<d$ and $b_1$ is an anti-atom of $\Sk(\mathbf{S})$ but no anti-atom of $\Sk(\mathbf{T})$. Applying \ref{AC4} to $b_1$ and $d$ yields a
  skeletal element $b_2$ such that $b_1<b_2<d$ and $b_1\bsup b_2^*<d$. Putting $\mathbf{S'}=\Sg{P}(
  S\cup\{b_2\})$ we obtain as for \eqref{eq_my_contr}
\begin{equation}\label{eq_my_contr2}
    \mathbf{S'}=\set{((s\w b_2)\bsup(t\w b_2^*))\w d: s,t\in \Sk(\mathbf{S}),d\in \D(\mathbf{S})},
\end{equation}
whose right hand side is isomorphic to
$\prod_{i=1}^{q-1}\widehat{\mathbf{F}_{f(i)}}\times\widehat{\mathbf{F}_{f(q)+1}}$ and thus to $\mathbf{T}$. Therefore there
is a skeletal anti-atom $\bar{b}\in T\setminus S$ such that $b_1<\bar{b}<d$ and $b_1\bsup\left.\bar{b}\right.^*<d$.

Now there is according to (\ref{eq_my_contr2}) a isomorphism $h\colon S'\to T$ over $S$ defined by
\[
h(((s\w b_2)\bsup(t\w b_2^*))\w d)=((s\w\bar{b})\bsup(t\w\left.\bar{b}\right.^*))\w d.
\]
\end{proof}
\begin{lemma}\label{lemma_spec_ext_case3}
Let $\mathbf{P}$ and $\mathbf{Q}$ be p-semilattices, $\mathbf{Q}$ an extension of $\mathbf{P}$, let $\mathbf{S}\cong\mathbf{2}^{p}\times\prod_{i=1}^q\widehat{\mathbf{F}_{f(i)}}$ be a finite subalgebra of $\mathbf{P}$ with $0\leq p$ and  $1\leq f(i)$ for $1\leq i\leq q$, and let $\mathbf{T}\cong\mathbf{2}^{p+1}\times\prod_{i=1}^q\widehat{\mathbf{F}_{f(i)}}$ be a finite subalgebra of $\mathbf{Q}$ that is an extension of $\mathbf{S}$.
 If $\mathbf{P}$ satisfies \ref{AC1}--\ref{AC4} and \ref{EC1}--\ref{EC5}, then
  then there is an extension $\mathbf{S'}$ of $\mathbf{S}$ in $\mathbf{P}$ satisfying $\mathbf{S'}\cong_S\mathbf{T}$.
\end{lemma}
\begin{proof}
We first consider the case $p>0$. In this case there is a unique
anti-atom $b_1$ of $\Sk(\mathbf{S})$ such that $b_1\parallel d$ for all $d\in
\D(\mathbf{S})\setminus\{1\}$ and $b_1$ is not an anti-atom of $\mathbf{T}$. Applying \ref{EC1} to $b_1$ and $1$ yields a
skeletal element $b_2$ such that $b_1<b_2<1$. Since
$\mathbf{T}\cong\mathbf{2}^{p+1}\times\prod_{i=1}^q\widehat{\mathbf{F}_{f(i)}}$ there
is a skeletal anti-atom $\bar{b}\in T\setminus S$ such that $b_1<\bar{b}<1$. Setting $\mathbf{S'}=\Sg{P}( S\cup\{b_2\})$ there
is a unique isomorphism $h\colon S'\ra T$ over $S$ and
$h(b_2)=\bar{b}$:

This holds because $b_2$ and $\bar{b}$ satisfy the same equations with respect to $\D(\mathbf{S})$ as $b_1$ and because there is a unique isomorphism 
\[
	h_1\colon\Sg{P}(\Sk(\mathbf{S})\cup\{b_2\})\ra \Sg{Q}(\Sk(\mathbf{S})\cup\{\bar{b}\})
\]
over $\Sk(\mathbf{S})$, see Remark \ref{remark_ec_ba}. 

We now consider the case $p=0$, that
is $\mathbf{T}\cong\mathbf{2}\times\prod_{i=1}^{q}\widehat{\mathbf{F}_{f(i)}}$. We proceed by the following steps:

(i) We describe how $S$ is embedded in $T$ and as a result the set $T\setminus S$.\\
(ii) We determine a Boolean element $\ol{b}\in T\setminus S$ such that $\Sg{P}\left(S\cup\{\ol{b}\}\right)=T$.\\
(iii) Applying axiom \ref{EC2} yields a Boolean element $b$, which behaves with respect to $\mathbf{S}$ in the same way as $\ol{b}$.\\
(iv) We determine the subalgebra $\mathbf{S'}:=\Sg{P}(S\cup\{b\})$, which will be shown to be isomorphic to $\mathbf{T}$ over $S$. \\
(v) We define a map $h\colon S'\ra T$ that is over $S$. We will show that $h$ is an isomorphism.\\
(v.i) $h$ is a homorphism.\\
(v.ii) $h$ is injective.

(i): Again we may assume
$\mathbf{T}=\prod_{i=0}^{q}\widehat{\mathbf{F}_{f(i)}}$ with $\widehat{\mathbf{F}_{f(0)}}:=\widehat{\mathbf{F}_{0}}=\mathbf{2}$,
identifying the subalgebra $\mathbf{T}$ of $\mathbf{Q}$ with the direct product $\mathbf{T}$ is
isomorphic to. There is an atom $a_{i,j}$ of $\widehat{\mathbf{F}_{f(i)}}$ with $i\in\{1,\ldots,q\}$ and $j\in\{1,\ldots,f(i)\}$, such that
\begin{multline}\label{eq_dir_prod_sa_3}
    S=\{x\in T\colon\left(\pi_i(x)\geq a_{i,j}\longrightarrow
\pi_0(x)=1\right)\And\\ \left(\pi_i(x)\not\geq a_{i,j}\longrightarrow \pi_0(x)=0\right)\}.
\end{multline}
We may assume $i=q$ and $j=1$. For
  $\bar{b}:=(0,1,\ldots,1)\in T\setminus S$ we have $\bar{b}\parallel d$ and $\left.\bar{b}\right.^*< d$
 for all $d\in \D(\mathbf{T})\setminus\{1\}$. We obtain %
\begin{multline}\label{eq_char_T}
    T=S\cup\set{d\w \bar{b}\w s: d\in\D(\mathbf{S}),s\in\Sk(\mathbf{S}),\pi_0(s)=1}\cup\\
    \set{d\w\left(\bar{b}\w s\right)^*: d\in\D(\mathbf{S}),s\in\Sk(\mathbf{S}),\pi_0(s)=1}
\end{multline}
 as follows: From \eqref{eq_dir_prod_sa_3} it follows %
\begin{multline}\label{eq_T_without_S}
    T\setminus S=\{x\in T\colon \left(\pi_q(x)\geq a_{q,1}\longrightarrow
\pi_0(x)=0\right)\And\\ \left(\pi_q(x)\not\geq a_{q,1}\longrightarrow \pi_0(x)=1\right)\}.
\end{multline}
Let $x\in T\setminus S$ be such that $\pi_q(x)\not\geq a_{q,1}$ and $\pi_0(x)=1$. There is  $d_x\in\D(\mathbf{T})=\D(\mathbf{S})$ such that
$x=d_x\w x^{**}$. For $t:=x^{**}$ due to \eqref{eq_dir_prod_sa_3}, as $t\not\in S$ follows from $x\not\in S$, we have $\pi_0(t)=1$ and $\pi_q(t)\not\geq a_{q,1}$. For $u\in T$ such that $\pi_0(u)=0$
and $\pi_k(u)=\pi_k(t)$ for $k=1,\ldots,q$ we have $u\in\Sk(\mathbf{S})$ according to \eqref{eq_dir_prod_sa_3}. Setting  $s=u^*$
we obtain $t=\left.\bar{b}\right.^*\bsup u=\left(\bar{b}\w u^*\right)^*=\left(\bar{b}\w s\right)^*$,
thus $x=d_x\w t=d_x\w\left(\bar{b}\w s\right)^*$ such that $s\in S$ and $\pi_0(s)=1$. Similarly one shows that for
$x\in T\setminus S$ such that $\pi_i(x)\geq a_{q,1}$ and $\pi_0(x)=0$ there is $s\in\Sk(\mathbf{S})$ such that $\pi_0(s)=1$ and
$d\in\D(\mathbf{S})$ such that $x=d\w s\w\bar{b}$. Obviously, the right hand side of \eqref{eq_char_T} is a
disjoint union.

(ii): Now we are going to show that there is a skeletal element $b\in P$ that behaves with respect
to $\mathbf{S}$ in the same way as $\bar{b}$. In order to express what this means, we define $a_{m}\in S$ to be the
maximal central element below the maximal dense element $d_m$ for $1\leq m\leq q$. Therefore, $\pi_k(d_m)=e$ if and only if $m=k$,
and 
\[
\pi_k(a_{m}) =
 \begin{cases}
1, & \text{for $k\not=m$;} \\
 0, & \text{for $k=m$;}
 \end{cases}\quad(m\not= q)
 \qquad
\pi_k(a_{q})=
            \begin{cases}
1, & \text{for $k\not\in\{0,q\}$;}  \\
 0, & \text{for $k\in\{0,q\}$.}
 \end{cases}
 \]
Furthermore, we have
%
\begin{gather}\label{eq_char_new_sk0}
   a_{q}=\Bsup\set{a_{m}^*: 1\leq m\leq q-1},\\
\label{eq_char_new_sk1}
   \bar{b}\parallel d_m\And \left.\bar{b}\right.^*<a_{m}\text{ for } m\in\{1,\ldots,q-1\},\\
\label{eq_char_new_sk2}
   a_{q}<\bar{b}\And  \bar{b}\w a_{q}^*\parallel d_q \And\left.\bar{b}\right.^*\bsup\, a_{q}<d_q. \end{gather}
(iii): Define $s_0=\Bsup\set{s\in
\Sk(\mathbf{S}): \pi_0(s)=0}$ and let $b$ be the result of applying \ref{EC2} to $a_{q}$, $s_0$
and $d_{q}$.
Then \eqref{eq_char_new_sk1} and \eqref{eq_char_new_sk2} are
satisfied if $\bar{b}$ is substituted by $b$: \eqref{eq_char_new_sk1} follows from $d_m\parallel a_{m}^*<a_{q}\leq s_0<b$, the first inequality being implied by \eqref{eq_char_new_sk0}. $b$ satisfies
\eqref{eq_char_new_sk2}, as $b$ is obtained by  applying \ref{EC2}
to $a_{q}$, $s_0$ and $d_{q}$. We additionally  have
\begin{equation}\label{eq_new_elem_c0}
    (\forall s\in S)(\pi_0(s)=0\longrightarrow s< b),
\end{equation}
dropping the assumption $s\in\Sk(\mathbf{S})$: There is $d_s\in\D(\mathbf{S})$ with $s=d_s\w s^{**}$. $\pi_0(s)=0$ implies $\pi_0(s^{**})=0$. By $s_0<b$ we obtain 
\[
s\w b=(d_s\w s^{**})\w b=d_s\w (s^{**}\w b)=d_s\w s^{**}=s.
\]

(iv): Now we show that for $\mathbf{S'}:=\Sg{P}( S\cup\{b\})$ there is an
isomorphism $h\colon T\ra S'$ over $S$ with $h\left(\bar{b}\right):=b$. We first describe $S'$, the carrier set of $\mathbf{S'}$:
\begin{multline}\label{eq_char_S'}
    S'=S\cup\set{d\w b\w s: d\in\D(\mathbf{S}),s\in\Sk(\mathbf{S}),\pi_0(s)=1}\cup\\
    \set{d\w (b\w s)^*: d\in\D(\mathbf{S}),s\in\Sk(\mathbf{S}),\pi_0(s)=1}.
\end{multline}
That rhs\eqref{eq_char_S'} is contained in $S'$ and that rhs\eqref{eq_char_S'} contains $S\cup\{b\}$ is obvious. For the converse we have to show that rhs\eqref{eq_char_S'} is closed under the operations.
We consider the cases that are not obvious. In the sequel we assume $d\in\D(\mathbf{S})$ and $s\in \Sk(\mathbf{S})$ with $\pi_0(s)=1$.
\begin{alignat*}{2}
 \left(d\wedge (b\w s)^*\right)^*& =\left((b\w s)^{*}\right)^* && \qquad \text {by \eqref{equation_dense_star}}\\
& =  b\w s && \qquad \text{by \eqref{equation_two_stars}}\\
 & = 1\w b\w s &&
\end{alignat*}

%
and similarly $\left(d\wedge (b\w s)\right)^* = 1\w(b\w s)^{*}$.
\begin{align*}
 (d_1\w (b\w s_1)^*)\w (d_2\w(b\w s_2)^*)& =d_1\w d_2\w \left(\left(b\w s_1\right)\bsup\left(b\w s_2\right)\right)^* \\
 & = d_1\w d_2\w\left(b\w(s_1\bsup s_2)\right)^*,
\end{align*}
$d_1\w d_2\w\left(b\w(s_1\bsup s_2)\right)^*\in S'$ as we have $\pi_0\left(s_1\bsup s_2\right)=1$.
\begin{alignat*}{2}
 (d_1\w (b\w s_1)^*)\w (d_2\w(b\w s_2))& =d_1\w d_2\w \left(b^*\bsup s_1^*\right)\w b\w s_2 &&\\
& = d_1\w d_2\w (s_1^*\w b)\w s_2 &&\\
 & = d_1\w d_2\w s_1^*\w s_2 && \text{by \eqref{eq_new_elem_c0}}\\
 & \in S &&
\end{alignat*}
Finally, we look at $x\in S$ and show that $x\w d\w(b\w s)$ and $x\w d\w(b\w s)^*$ are also contained in rhs\eqref{eq_char_S'}. First we consider $x\w d\w(b\w s)$. If $\pi_0(x)=1$ then $x\w d\w(b\w s)$ is contained in rhs\eqref{eq_char_S'} since $\pi_0(x\w s)=1$. If $\pi_0(x)=0$ then $x\w d\w(b\w s)=x\w d\w s\in S$ by \eqref{eq_new_elem_c0}. Next we consider $x\w d\w(b\w s)^*$. There is $d_x\in\D(\mathbf{S})$ with $x=d_x\w x^{**}$. First we assume $\pi_0(x)=0$, which implies $x^*\bsup b=1$.
\begin{alignat*}{2}
 x\w (d\w(b\w s)^*) &=  d\w d_x\w x^{**}\w(b\w s)^* && \\
  &=  d\w d_x\w
  \left(x^*\bsup(b\w s)\right)^*&&  \\
  & = d\w d_x\w \left((x^*\bsup b)\w (x^*\bsup s)\right)^*&&\\
&= d\w d_x\w \left(x^*\bsup s\right)^*&&\text{by $x^*\bsup b=1$}  \\
&\in S &&
\end{alignat*}
 Now let $\pi_0(x)=1$.
\begin{alignat*}{2}
 x\w (d\w(b\w s)^*) &= d\w d_x\w x^{**}\w(b\w s)^* && \\
 &= d\w d_x\w\left(x^*\bsup(b\w s)\right)^*  &&\\
   &= d\w d_x\w\left((x^*\bsup b)\w(x^*\bsup s)\right)^* && \\
  &=  d\w d_x\w \left(b\w(x^*\bsup s)\right)^* && \text{by \eqref{eq_new_elem_c0} and $\pi_0\left(x^{*}\right)=0$}
\end{alignat*}

Note that $a_{q}$ is the only maximal central element of $\mathbf{S}$ that is not a maximal skeletal element of $\mathbf{S'}$ anymore.
In $\mathbf{S'}$ we have $a_{q}<b^*\bsup a_{q}=(b\w a_{q}^*)^*<d_q$.

(v): As rhs\eqref{eq_char_T} is a disjoint union
\[
h(x):=
      \begin{cases}
      x, & x\in S;\\
      &\\
        d\w b\w s, &  x=d\w \bar{b}\w s, s\in\Sk(\mathbf{S}), \pi_0(s)=1,\\
        & d\in\D(\mathbf{S}); \\
        d\w\left(b\w s\right)^* & x=d\w \left(\bar{b}\w s\right)^*,s\in\Sk(\mathbf{S}),\pi_0(s)=1,\\
        & d\in\D(\mathbf{S})
      \end{cases}
    \]
is well-defined. Obviously, $h$ is over $S$. \eqref{eq_char_S'} implies that $h$ is onto $S'$.

It remains to show that for all $u,v\in T$
\begin{align}\label{hom_inf}
    h(u\w v)&=h(u)\w h(v)\\
\label{hom_compl}
h(u^*)&=h(u)^*
\end{align}
hold and that $h$ is injective.

(v.i): For \eqref{hom_inf} we consider, assuming $\pi_0(s_u)=\pi_0(s_v)=1$, four cases.
%

(1): $u=d_u\w (\bar{b}\w s_u)^*$, $v=d_v\w(\bar{b}\w s_2)^*$.
    \begin{align*}
      h(u\w v) &=  h\left((d_u\w (\bar{b}\w s_u)^*)\w (d_v\w(\bar{b}\w s_v)^*)\right)  \\
       & =   h\left(d_u\w d_v\w((\bar{b}\w s_u)\bsup(\bar{b}\w s_v) )^*\right)\\
       &=   h\left(d\w \left(\bar{b}\w (s_u\bsup s_v )\right)^*\right)\\
       &= d_u\w d_v\w \left(b\w (s_u\bsup s_v )\right)^* \\
       & =   d_u\w d_v\w \left( (b\w s_u)\bsup (b\w s_v)\right)^* \\
       & =   (d_u\w (b\w s_u)^*)\w (d_v\w (b\w s_v )^*) \\
       & =  h(u)\w h(v)
    \end{align*}

(2): $u=d_u\w \bar{b}\w s_u$, $v=d_v\w(\bar{b}\w s_v)^*$.
    \begin{alignat*}{2}
      h(u\w v) &=  h\left((d_u\w \bar{b}\w s_u)\w (d_v\w(\bar{b}\w s_v)^*)\right) && \\
       & =   h\left(d_u\w d_v\w \bar{b}\w s_u\w\left(\left.\bar{b}\right.^*\bsup s_v^* \right)\right) &&\\
       &=   h\left(d\w s_u\w\left(\left(\bar{b}\w\left.\bar{b}\right.^*\right)\bsup\left(\bar{b}\w s_v^*\right) \right)\right) && \text{by $d:=d_u\w d_v$}\\
       &= h\left(d\w s_u\w \bar{b}\w s_v^* \right) && \\
       &= h\left(d\w s_u\w s_v^* \right) && \text{by $\bar{b}>s_v^*$}\\
       & = d\w s_u\w s_v^* &&\\
       & =  d\w s_u\w \left(b\w \left(b^*\bsup s_v^*\right)\right) && \text{by \eqref{eq_new_elem_c0}} \\
       & =   (d_u\w b\w s_u)\w (d_v\w\left(b\w s_v\right)^*) &&\\
       & =  h(u)\w h(v) &&
    \end{alignat*}

  (3):
   $u\in S$, $v=d\w \bar{b}\w s$ with $\pi_0(s)=1$. We consider two subcases:

  (3.1): $\pi_0(u)=1$. Then $\pi_0(u\w s)=1$, thus
    \begin{align*}
      h(u\w v) &= h(u\w (d\w\bar{b}\w s)) \\
       &= h(d\w \bar{b}\w (u\w s)) \\
       &= d\w b\w (u\w s)\\
       &= u\w (d\w b\w s)\\
       &= h(u)\w h(v)
    \end{align*}

    (3.2): $\pi_0(u)=0$.
    \begin{alignat*}{2}
      h(u\w v) &= h(u\w (d\w\bar{b}\w s)) &&\\
       &= h(d\w u\w s) && \text{by $\bar{b}>u$} \\
       &= d\w u\w s && \\
       &= u\w (d\w b\w s) && \text{by \eqref{eq_new_elem_c0}}\\
       &= h(u)\w h(v) &&
    \end{alignat*}

   (4): $u\in S$, $v=d\w (\bar{b}\w s)^*$ with $\pi_0(s)=1$. There is $d_u\in\D(\mathbf{S})$ such that $u=d_u\w u^{**}$.
   We consider again two subcases:

  (4.1): $\pi_0(u)=1$. Then $\pi_0(u\w s)=1$, thus
    \begin{alignat*}{2}
      h(u\w v) &= h\left((d_u\w u^{**})\w \left(d\w\left(\bar{b}\w s\right)^*\right)\right) &&\\
       &= h\left(d_u\w d\w \left(u^*\bsup\left(\bar{b}\w s\right)\right)^*\right) &&\\
       &= h\left(d_u\w d\w \left(\left(u^*\bsup\bar{b}\right)\w
       \left(u^*\bsup s\right)\right)^*\right) &&\\
       &= h\left(d_u\w d\w \left(\bar{b}\w
       \left(u^*\bsup s\right)\right)^*\right)&& \text{by $\bar{b}>u^*$}\\
       & = d_u\w d\w \left(b\w \left(u^*\bsup s\right)\right)^* &&\\
       & = d_u\w d\w \left(\left(u^*\bsup b\right)\w \left(u^*\bsup s\right)\right)^* && \text{by \eqref{eq_new_elem_c0}} \\
       &= d_u\w d\w \left(u^*\bsup (b\w s)\right)^* &&\\
       & =  (d_u\w u^{**})\w (d\w\left(b\w s\right)^*) && \\
       &= h(u)\w h(v). &&
    \end{alignat*}

    (4.2): $\pi_0(u)=0$:
    \begin{alignat*}{2}
      h(u\w v) &= h\left((d_u\w u^{**})\w \left(d\w\left(\bar{b}\w s\right)^*\right)\right) &&\\
      & =  h\left(d_u\w d\w u^{**}\w\left(\left.\bar{b}\right.^*\bsup s^*  \right)\right) &&\\
       & = h\left(d_u\w d\w \left( \left(u^{**}\w \left.\bar{b}\right.^*\right)\bsup(u^{**}\w s^*)  \right)\right) &&\\
       &= h(d\w d_u\w u^{**}\w s^*) && \text{by $u^{**}\w \left.\bar{b}\right.^*=0$}\\
       &= d\w d_u\w u^{**}\w s^*&&\\
       &= d\w d_u\w \left((u^{**}\w b^*)\bsup(u^{**}\w s^*)\right) && \text{by \eqref{eq_new_elem_c0}}\\
       &= d_u\w d\w \left(u^{**}\w(b^*\bsup s^*)\right)&&\\
       &= u\w (d\w(b\w s)^*)&&\\
       &= h(u)\w h(v)
    \end{alignat*}
For \eqref{hom_compl} we consider, assuming
$\pi_0(s)=1$, the following cases: \\
(1): $u=d\w \bar{b}\w s$:
    \begin{alignat*}{2}
      h(u^*) &= h\left(\left(d\w \bar{b}\w s\right)^*\right) &&\\
        &= h\left(1\w\left(\bar{b}\w s\right)^*\right) && \text{by \eqref{equation_dense_star}} \\
        &= 1\w(b\w s)^* &&\\
        &= (d\w b\w s)^* &&\\
        &= h(u)^* &&
    \end{alignat*}
 (2): $u=d\w (\bar{b}\w s)^*$:
    \begin{alignat*}{2}
      h(u^*) &= h\left(\left(d\w \left(\bar{b}\w s\right)^*\right)^*\right) &&\\
       &= h\left(1\w \left(\bar{b}\w s\right)^{**}\right)&&\text{by \eqref{equation_dense_star}}\\
       &=  h\left(1\w \bar{b}\w s\right)&&\\
       &= 1\w b\w s &&\\
       &= (d\w(b\w s)^*)^* &&\\
       &= h\left(d\w\left(\bar{b}\w s\right)^*\right)^* && \\
       &= h(u)^* &&
    \end{alignat*}

(v.ii): To show the injectivity of $h$ assume $x,y\in T$ with $x\not=y$. If $x,y\in S$ then $h(x)\not=h(y)$ trivially holds. We consider the following non-trivial cases:

(1): $x\in S$, $y\in T\setminus S$. We consider the following subcases:

(1.1): $y=d_y\w \bar{b}\w s_y$,
       $\pi_0(x)=0$. Then $h(x)=h(y)$ is impossible:
      %
      \begin{alignat*}{2}
        h(x)= h(y) &\implies x=d_y\w b\w s_y &&\\
        &\implies x^{**}=b\w s_y&&\\
        &\implies a_{q}^*\bsup x^{**} = a_{q}^*\bsup\left(b\w s_y\right) &&\\
        &\implies a_{q}^*\bsup x^{**} = \left(a_{q}^*\bsup b\right)\w\left(a_{q}^*\bsup s_y\right)&& \\
        &\implies a_{q}^*\bsup x^{**}=a_{q}^*\bsup s_y && \text{by $a_{q}^*\bsup b=1$}\\
        &\implies \pi_q\left(a_{q}^*\bsup x^{**}\right)=\pi_q\left(a_{q}^*\bsup s_y\right) &&\\
        &\implies \pi_q\left(x^{**}\right)=\pi_q\left(s_y\right) &&
      \end{alignat*}
      But as $\pi_0(x)=0$ and $\pi_0(s_y)=1$ we have $\pi_0(x)\not\geq a_{q,j}$, $\pi_q(s_y)\geq a_{q,j}$, contradicting the preceding equality.

(1.2): $y=d_y\w \bar{b}\w s_y$, $\pi_0(x)=1$. Then $h(x)=h(y)$ again implies $x^{**}=b\w s_y$ from which we obtain $x^{**}\leq b$. Furthermore, $x^*<b$ from \eqref{eq_new_elem_c0} since $\pi_0(x^*)=0$. The last two inequalities imply $b=1$ contradicting the choice of $b$.

(1.3): $y=d_y\w \left(\bar{b}\w s_y\right)^*$, $\pi_0(x)=0$.  $h(x)=h(y)$ is impossible:
      Similarly to the preceding subcase we obtain $x^{*}\leq b$. But \eqref{eq_new_elem_c0} and $\pi_0(x)=0$ imply $x\leq b$. Together we obtain $b=1$ again contradicting the choice  of $b$.

(1.4): $y=d_y\w \left(\bar{b}\w s_y\right)^*$, $\pi_0(x)=1$. $h(x)=h(y)$ is impossible:
      \begin{alignat*}{2}
        h(x)= h(y) &\implies x=d_y\w (b\w s_y)^* &&\\
        &\implies x^{**}=(b\w s_y)^*&&\\
        &\implies x^{**} =b^*\bsup s_y^* &&\\
        &\implies b\w x^{**}=b\w s_y^*&&\\
        &\implies b^*\bsup x^{*} =b^*\bsup s_y &&\\
        &\implies a_{1}\bsup x^{*} =a_{1}\bsup s_y && \text{by \eqref{eq_char_new_sk1} and $m=1$}\\
        &\implies  \pi_q(a_{1}\bsup x^{*}) =\pi_q(a_{1}\bsup s_y) &&\\
        &\implies \pi_q(x^{*})=\pi_q(s_y) &&
        \end{alignat*}
      But the last equation contradicts $\pi_q(x^*)\not\geq a_{q,j}$, $\pi_q(y)\geq a_{q,j}$.

(2): $x,y\in T\setminus S$. We consider the following subcases:

(2.1): $x=d_x\w \bar{b}\w s_x$, $y=d_y\w \bar{b}\w s_y$.  Then $h(x)=h(y)$ implies $b\w s_x= b\w s_y$. As $\pi_0(s_x^*)=\pi_0(s_y^*)=0$ \eqref{eq_new_elem_c0} implies $s_x^*,s_y^*<b$, thus $b^*\leq s_x,s_y$, from which we obtain $b^*\w s_x=b^*\w s_y$. It follows $s_x=s_y$.

      $d_x\w \bar{b}\w s_x\not=d_y\w \bar{b}\w s_y$ is not possible: Because of $\pi_0(\bar{b})=0$ there is, setting $s=s_x=s_y$,
      $m\in\{1,\ldots,q\}$ such that $\pi_m(d_x)=e$ and $\pi_m(d_y)=1$, which is equivalent to $a_{m}^*\w d_x<a_{m}^*\w d_y=a_{m}^*\w s$. 
      In the case $m<q$ we have $a_{m}^*<b$ due to \eqref{eq_char_new_sk1}, thus $d_x\w b\w s_x\not=d_y\w b\w s_y$. In the case $m=q$ we have $b\w a_{q}^*\parallel d_q$, which is \eqref{eq_char_new_sk2}. Furthermore, $s\geq a_{q}^*$ as $\pi_0(s)=\pi_q(s)=1$. We obtain $h(y)=d_y\w b\w s\parallel d_q$. On the other hand because of $d_x\leq d_q$
      we have $h(x)\leq d_x\leq d_q$, again contradicting our assumption $h(x)=h(y)$.

(2.2): $x=d_x\w \left(\bar{b}\w s_x\right)^*$, $y=d_y\w \left(\bar{b}\w s_y\right)^*$. As in the preceding subcase $h(x)=h(y)$ implies $b\w s_x=b\w s_y$, again leading to a contradiction.

(2.3): $x=d_x\w \bar{b}\w s_x$, $y=d_y\w \left(\bar{b}\w s_y\right)^*$. Here $h(x)= h(y)$ implies $b\w s_x= b^*\bsup s_y^*$, which is impossible.
\end{proof}

Now we can prove that the satisfiability of the axioms \ref{AC1}--\ref{AC4} and \ref{EC1}--\ref{EC5} is sufficient for a p-semilattice to be existentially closed.
\begin{theorem}\label{theorem_sufficiency}
If a p-semilattice $\mathbf{P}$ satisfies the axioms \ref{AC1}--\ref{AC4} and \ref{EC1}--\ref{EC5}, then it is existentially closed.
\end{theorem}

\begin{proof} 
We assume that $\mathbf{S}$, $\mathbf{T}$ and $\mathbf{Q}$ satisfy \ref{lemma_criterion_ec_item1} and \ref{lemma_criterion_ec_item2} of Lemma \ref{lemma_criterion_ec} and show that there is $\mathbf{S'}\leq\mathbf{P}$ such that $\mathbf{S'}\cong_S\mathbf{T}$. With Lemma \ref{lemma_criterion_ec} follows that Lemma $\mathbf{P}$ is existentially closed.

 As $\mathbf{S}\leq\mathbf{T}$ the number $r$ of Boolean anti-atoms of $\mathbf{S}$ that are also anti-atoms of $\mathbf{T}$ is less or equal than $r'$, the number of Boolean anti-atoms of $\mathbf{T}$.  $k-r=\sum_{i=1}^{s_1} f_1(i)$ is the number of Boolean anti-atoms of $\mathbf{S}$ that are below a proper dense element of $\mathbf{T}$; the factor $\prod_{i=1}^{s_1}\mathbf{F}_{f_1(i)}$ in \eqref{eq_apprex} generates the Boolean anti-atoms of $\mathbf{S}$ that are not anti-atoms of $\mathbf{T}$.

If $s'=0$ then $s_1=s_2=0$; applying \ref{EC1} $r'-r$ times yields a subalgebra $\overline{\mathbf{S}}$ of $\mathbf{P}$ satisfying $\overline{\mathbf{S}}\cong_S\mathbf{T}$.
Therefore we assume $0\leq s_1+s_2<s'$ and $0\leq r\leq r'$.

According to Lemma \ref{lemma_stepwise_ext0} there is a sequence
$\mathbf{T}_0,\ldots,\mathbf{T}_{s_1}$ of subalgebras of $\mathbf{T}$ with $\mathbf{T}_0=\mathbf{S}$ such that for
$k=0,\ldots,s_1-1$ we have $\mathbf{T}_k\leq \mathbf{T}_{k+1}$ and
\begin{equation}\label{eq_suff_ext1}
    \mathbf{T}_{k}\cong\mathbf{2}^r\times\prod_{i=1}^{k}\widehat{\mathbf{F}_{f_1(i)}}\times\prod_{i=k+1}^{s_1}\mathbf{F}_{f_1(i)}\times\prod_{i=1}^{s_2}\widehat{\mathbf{F}_{f_2(i)}},
\end{equation}
thus,
\[
  \mathbf{T}_{s_1}\cong\mathbf{2}^r\times\prod_{i=1}^{s_1}\widehat{\mathbf{F}_{f_1(i)}}\times\prod_{i=1}^{s_2}\widehat{\mathbf{F}_{f_2(i)}},
\]
which we can write as
\begin{equation}\label{eq_suff_ext2}
  \mathbf{T}_{s_1}\cong\mathbf{2}^r\times\prod_{i=1}^{s_1+s_2}\widehat{\mathbf{F}_{f(i)}}
\end{equation}
with  $f(i)=f_1(i)$ if $1\leq i\leq s_1$ and $f(i)=f_2(i)$ if $s_1+1\leq i\leq s_1+s_2$.
%

For all $i\in\{1,\ldots,s_1+s_2\}$ there is a sequence $\mathbf{T}_{i,0},\ldots,\mathbf{T}_{i,g(i)-f(i)}$ such that
\begin{align}\label{eq_suff_ext2a0}
   \mathbf{T}_{i,j}&\leq \mathbf{T}_{i,j+1}\quad(0\leq j<g(i)-f(i)),\\
   \label{eq_suff_ext2a1}
   \mathbf{T}_{i,j_1}&\leq \mathbf{T}_{i+1,j_2}\quad(0\leq j_1<g(i)-f(i),0\leq j_2<g(i+1)-f(i+1)),\\
\label{eq_suff_ext2a}
    \mathbf{T}_{i,j}&\cong\mathbf{2}^r\times\prod_{k=1}^{i-1}\widehat{\mathbf{F}_{g(k)}}\times\wh{\mathbf{F}_{f(i)+j}}\times\prod_{k=i+1}^{s_1+s_2}\widehat{\mathbf{F}_{f(k)}}\quad(0\leq j\leq g(i)-f(i)),
\end{align}
%
thus
\begin{equation}\label{eq_zw_schritt3}
    \mathbf{T}_{s_1+s_2,g(s_1+s_2)-f(s_1+s_2)}\cong\mathbf{2}^r\times\prod_{i=1}^{s_1+s_2}\widehat{\mathbf{F}_{g(i)}}.
\end{equation}
According to Lemma \ref{lemma_stepwise_ext1} there is, setting $q=s'-s_1-s_2$, a sequence
$\mathbf{U}_0,\ldots,\mathbf{U}_{2q}$ of subalgebras of $\mathbf{T}$ with $\mathbf{U}_0=\mathbf{T}_{s_1+s_2,g(s_1+s_2)-f(s_1+s_2)}$ and
$\mathbf{U}_{2q}\cong\mathbf{2}^r\times\prod_{i=1}^{s'}\widehat{\mathbf{F}_{g(i)}}$ such that for
$k=0,\ldots,2q-1$ we have $\mathbf{U}_k\leq \mathbf{U}_{k+1}$, whereby
$\mathbf{U}_{k+1}\cong \mathbf{U}_k\times\widehat{\mathbf{F}_{l_{k+1}}}$ ($k=0,\ldots,q-1\text{ and }1\leq l_{k+1}\leq g(k+1)$),
\begin{equation}\label{eq_suff_ext2}
    \mathbf{U}_{q+k}\cong\mathbf{2}^r\times\prod_{i=1}^{s_1+s_2+k}\widehat{\mathbf{F}_{g(i)}}\times\prod_{i=s_1+s_2+k+1}^{s'}\widehat{\mathbf{F}_{l_i}}\quad (k=0,\ldots,q).
\end{equation}
%
%
In \eqref{eq_suff_ext2} there is for every $k\in\{0,\ldots,q-1\}$ a
sequence $\mathbf{U}_{q+k,0},\allowbreak\ldots,\mathbf{U}_{q+k,g(k)-l_k}$ such that
\begin{align}\label{eq_suff_ext2b0}
   \mathbf{U}_{q+k,j}&\leq \mathbf{U}_{q+k,j+1}\quad(0\leq j<g(k)-l_k),\\
\label{eq_suff_ext2b}
    \mathbf{U}_{q+k,j}&\cong\mathbf{2}^r\times\prod_{i=1}^{s_1+s_2+k-1}\widehat{\mathbf{F}_{g(i)}}\times\wh{\mathbf{F}_{l_k+j}}\\\nonumber
    &\times\prod_{i=s_1+s_2+k+1}^{s'}\widehat{\mathbf{F}_{l_i}}\quad(0\leq j\leq g(k)-l_k).
\end{align}
%
%
%
Finally, there is according to Lemma \ref{lemma_stepwise_ext2} a
sequence $\mathbf{V}_0,\ldots,\mathbf{V}_{r'-r}$ of subalgebras of $\mathbf{T}$ such that $\mathbf{V}_{j}\leq \mathbf{V}_{j+1}\text{ for }0\leq j<r'-r$ and
\begin{equation}\label{eq_suff_ext3}
    \mathbf{V}_{j}\cong\mathbf{2}^{r+j}\times \prod_{i=1}^{s'}\widehat{\mathbf{F}_{g(i)}}\qquad (j=0,\ldots,r'-r).
\end{equation}
We set $\mathbf{S}_0=\mathbf{S}$ and $h_0=\text{id}_S$. According to Lemma \ref{lemma_spec_ext_case0} there exists for
every $k\in\{0,\ldots,s_1-1\}$ a subalgebra $\mathbf{S}_{k+1}$ of $\mathbf{P}$ and an
isomorphism $h_{k+1}\colon S_{k+1}\to T_{k+1}$ extending $h_k$,
$\left(\mathbf{T}_k\right)_{0\leq k\leq s_1}$ the above sequence of subalgebras of $\mathbf{T}$ satisfying \eqref{eq_suff_ext1}.

Now we set $\mathbf{S}_{0,0}=\mathbf{S}_{s_1}$. According to Lemma \ref{lemma_spec_ext_case2} there exists for
every $i\in\{1,\ldots,s_1+s_2\}$ and every $j\in\{0,\ldots,g(i)-f(i)\}$
a subalgebra $\mathbf{S}_{i,j+1}$ and an isomorphism $h_{i,j+1}\colon S_{i,j+1}\to
T_{i,j+1}$ extending $h_{i,j}$, the above sequences of subalgebras
$\left(\mathbf{T}_{i,j}\right)_{0\leq j\leq g(i)-f(i)}$ of $\mathbf{T}$
satisfying \eqref{eq_suff_ext2a0}-\eqref{eq_suff_ext2a}.

Now we set $\mathbf{S}_{q}=\mathbf{S}_{s_1+s_2,g(s_1+s_2)-f(s_1+s_2)}$.
 According to Lemma \ref{lemma_spec_ext_case1} there exists for
every $k\in\{0,\ldots,q-1\}$ a subalgebra $\mathbf{S}_{q+k+1}$ of $\mathbf{P}$ and an
isomorphism $h_{q+k+1}\colon S_{q+k+1}\to U_{k+1}$ extending $h_{q+k}$,
$\left(\mathbf{U}_k\right)_{1\leq k\leq q}$ the above sequence of subalgebras of $\mathbf{T}$ satisfying \eqref{eq_suff_ext2b0} and \eqref{eq_suff_ext2b} respectively.
According to Lemma \ref{lemma_spec_ext_case2} there exists for
every $k\in\{1,\ldots,q\}$ and every $j\in\{0,\ldots,g(k)-l_k-1\}$
a subalgebra $\mathbf{S}_{q+k,j+1}$ and an isomorphism $h_{q+k,j+1}\colon S_{q+k,j+1}\ra
U_{k,j+1}$ extending $h_{q+k,j}$,
$\left(\mathbf{U}_{k,j}\right)_{0\leq j\leq g(k)-l_k-1}$ the above sequence of subalgebras of $\mathbf{T}$ satisfying \eqref{eq_suff_ext2b}.

According to Lemma \ref{lemma_spec_ext_case3} there exists for
every $j\in\{0,\ldots,r'-r\}$ a subalgebra $\mathbf{S}_{2q+j+1}$  of $\mathbf{P}$ and
an isomorphism $h_{2q+j+1}\colon\allowbreak S_{2q+j+1}\to V_{j+1}$ extending $h_{2q+j}$, $\left(\mathbf{V}_j\right)_{0\leq j\leq r'-r}$ the above sequence of subalgebras of $\mathbf{T}$ satisfying \eqref{eq_suff_ext3}.

The above implies that $h_{2q+r'-r}\colon S_{2q+r'-r}\to T$ is the desired isomorphism over $S$ since $\mathbf{V}_{r'-r}
=\mathbf{T}$ and every extension of $h_1$ is over $S$.
\end{proof}

With the preceding results we finally obtain the desired result.

\begin{corollary}\label{corollary_main}
A p-semilattice $\mathbf{P}$ is existentially closed if and only if $\mathbf{P}$ satisfies \ref{AC1}--\ref{AC4} and \ref{EC1}--\ref{EC5}.
\end{corollary}

\begin{proof}
Combine Theorem \ref{theorem_necessity} and Theroem \ref{theorem_sufficiency}.
\end{proof}




\end{document}